\documentclass[12pt]{article}

\usepackage[letterpaper, margin=1.7cm]{geometry}
\usepackage[normalem]{ulem}
\usepackage{amsmath, amsfonts, amssymb, amsthm}
\usepackage{setspace}
\usepackage{commath}
\usepackage{mathtools}

\usepackage{csquotes}
\usepackage{bbm}
\usepackage{bm}
\usepackage{subfig} 
\usepackage{color}
\usepackage[normalem]{ulem}
\usepackage{mathrsfs}
\usepackage{overpic}

\usepackage{enumerate}
\usepackage[singlelinecheck=false]{caption}
\usepackage{hyperref}
\usepackage{cleveref}

\def\ignore#1{}

\def\bR{{\mathbb R}}
\def\R{{\mathbb R}}

\def\bN{\mathbb N}
\def\bQ{\mathbb Q}

\def\bP{{\mathbb P}}
\def\bE{{\mathbb E}}

\def\E{{\mathbb E}}

\definecolor{DarkGreen}{rgb}{0.2,0.6,0.2}

\newcommand{\se}{ \subseteq}

\newcommand{\ee}{ \varepsilon }

\def\fp#1{\{#1\}}

\def\eps{\varepsilon}
\def\<{\langle}
\def\>{\rangle}

 \def\ua{\uparrow}

\newcommand{\kk}{ \kappa }

\newcommand{\bY}{ Y }
\newcommand{\bB}{ B }
\newcommand{\bW}{ W }
\newcommand{\bX}{ X }
\newcommand{\bZ}{ {\mathbb{Z}} }
\newcommand{\n}[1]{\left\lVert#1\right\rVert}
\newcommand{\bn}[1]{\bigg\lVert#1\bigg\rVert}
\newcommand{\nn}[1]{\lVert#1\rVert}
\newcommand{\lr}[2]{\left\langle#1,#2\right\rangle}
\newcommand{\supp}{ {\text{supp }} }
\newcommand{\card}{ {\#} }
\newcommand{\bone}{ {\mathbbm{1}} }
\newcommand{\var}[1]{\text{Var}\left(#1\right)}

\def\ignore#1{}

\def\<{\langle}\def\>{\rangle}

\newtheorem{theorem}{Theorem}[section]

\newtheorem{corollary}[theorem]{Corollary}
\newtheorem{lemma}[theorem]{Lemma}

\theoremstyle{definition}
\newtheorem{definition}[theorem]{Definition} 
 
\newtheorem{remark}[theorem]{Remark}

\begin{document}
\title{\Large\bf Sample Path Properties of the Fractional Wiener--Weierstrass Bridge}
\author{\normalsize Alexander Schied\thanks{Department of Statistics and Actuarial Science, University of Waterloo. E-mail: {\tt aschied@uwaterloo.ca}}
	 \and\setcounter{footnote}{6}
\normalsize	Zhenyuan Zhang\thanks{
		Department of Mathematics, Stanford University. E-mail: {\tt zzy@stanford.edu}
	 		\hfill\break The authors gratefully acknowledge financial support  from the
 Natural Sciences and Engineering Research Council of Canada through grants RGPIN-2017-04054 and RGPIN-2024-03761.}}
        
\maketitle

\vspace{-0.5cm}

\begin{abstract} 
Fractional Wiener--Weierstrass  bridges are a class of Gaussian processes that arise from replacing the trigonometric function in the construction of classical Weierstrass  functions by a fractional Brownian bridge. 
We investigate the sample path properties of such processes, including local and uniform moduli of continuity, $\Phi$-variation, Hausdorff dimension, and location of the maximum. Our analysis relies heavily on upper and lower bounds of fractional integrals, where we establish a novel improvement of the classical Hardy--Littlewood inequality for fractional integrals of a special class of step functions.\end{abstract}
 
\medskip


\maketitle


\section{Introduction}

In recent years, interest in non-diffusive stochastic models---those with sample paths either rougher or smoother than standard Brownian motion---has grown significantly. This trend is driven by several factors. First, advances in Lyons' rough path theory \cite{friz2020course}, modeling with fractional Brownian motion \cite{KubiliusMishura}, and pathwise It\^o calculus for trajectories with $p$-th
variation for 
$p>2$ \cite{ContPerkowski} have all spurred further research. Second, new applications of non-diffusive processes, such as rough volatility modeling \cite{RoughVolatilityBook}, have underscored their practical relevance.

The phenomenon of roughness has also played a significant role in fractal geometry. Consider, for instance, the classical Weierstrass  function, which, for 
$\alpha\in(0,1)$ and $b\in\{2,3,\dots\}$, is defined as 
\begin{align}
    \phi_{\alpha,b}(t)=\sum_{n=0}^\infty \alpha^n\cos(2\pi b^nt),\qquad t\in[0,1].\label{eq:w}
\end{align}
Recent analyses have studied this function in a manner analogous to the sample paths of stochastic processes. 
Among the properties explored are the local and uniform moduli of continuity \cite{de2019central}, the existence of a local time \cite{ImkellerWeierstrass}, the Hausdorff dimension of its graph \cite{keller2017simpler}, and its suitability as an integrator for pathwise Itô calculus \cite{SchiedZZhang}.

In this paper, we continue the analysis of a new class of stochastic processes introduced by the authors in \cite{schied2024weierstrass}, aiming to provide a synthesis between fractional Gaussian processes and fractal geometry. These processes are obtained by replacing the cosine function in \eqref{eq:w} by the trajectories of a fractional Brownian bridge $B_H$ with Hurst parameter $H$ (the precise definition of $B_H$ is given by \eqref{fbb} below). More precisely, the \emph{fractional Wiener--Weierstrass  bridge} with parameters $\alpha\in(0,1)$, $b\in\{2,3,\dots\}$, and $H\in(0,1)$ is defined as the stochastic process
 $$Y(t):=
\sum_{n=0}^\infty \alpha^nB_H(\fp{b^n t}),\qquad 0\le t\le 1,
 $$
where $\fp{x}$ is the fractional part of $x\ge0$. Although $Y$ remains a Gaussian process, it displays a number of intriguing properties. For instance, it was shown in \cite{schied2024weierstrass} that there is no combination of parameter values $(\alpha,b,H)$ for which $Y$ is a semimartingale. It was shown moreover that the covariance function $c(s,t):=\bE[Y(s)Y(t)]$ typically has a fractal structure and, for fixed $s\in(0,1)$, can have the same roughness as the sample paths of $Y$. Here, the roughness of a  function $f:[0,1]\to\bR$  is quantified through its roughness exponent along the $b$-adic partitions, which according to \cite{HanSchiedHurst} is defined as a number $R\in[0,1]$ for which
\begin{equation}\label{roughness exponent eq}
\lim_{n\ua\infty}\sum_{k=0}^{b^n-1}\big|f((k+1)b^{-n})-f(kb^{-n})\big|^p=\begin{cases}\infty&\text{if $p<1/R$;}\\
0&\text{if $p>1/R$.}
\end{cases}
\end{equation}
The main results of \cite{schied2024weierstrass} identify the roughness exponent $R$ of the sample paths of $Y$ and compute their $p$-th variation (i.e., the limit on the left-hand side of \eqref{roughness exponent eq}) for $p = 1/R$. These results show that two distinct regimes emerge, arising from the competition between the roughness exponents of the trajectories of the fractional Brownian bridge $B_H$ and the classical Weierstrass  function in \eqref{eq:w}, which are given by $H$ and $K:=1\wedge(-\log_b\alpha)$, respectively.

In this paper, we continue our analysis of the fractional Wiener–Weierstrass  bridge, focusing on several sample path properties: (Wiener–Young) $\Phi$-variation, the local and uniform moduli of continuity, the Hausdorff dimension of the graph, and the location of the maximum. Our results highlight the significance of the distinct regimes arising from the relationship between the roughness exponents $H$ and $K$.
Specifically, for $H<K$, we observe that the fine structure of the sample paths retains characteristics similar to those of $B_H$. In contrast, when $K<H$, the behavior of the trajectories of $Y$ resembles that of a randomized version of the classical Weierstrass  function. In this case, the limits of the local and uniform moduli of continuity are governed by non-degenerate random variables, providing an example for which the well-known zero-one law for the moduli of continuity of Gaussian processes fails (see Lemma 7.1.1 of \cite{MarcusRosen} or \Cref{lemma:0-1} below).
The critical case $H=K$ presents the most intriguing and challenging regime. Here, the fine structure of the fractional Wiener–Weierstrass  bridge deviates from both the classical Weierstrass  function and the trajectories of $B_H$.

There is extensive literature on the sample path properties of Gaussian processes; see, for example, the books by Adler \cite{adler2010geometry} and Marcus and Rosen \cite{MarcusRosen}. However, most proofs in this area rely on the stationarity of increments and explicit knowledge of the covariance function. In contrast, the fractional Wiener–Weierstrass  process has highly non-stationary increments, a fractal covariance structure, and lacks self-similarity, posing significant challenges for our analysis. To handle the complex covariance structure of $Y$, we establish upper and lower bounds on $\bE[(Y(t)-Y(s))^2]$ for sufficiently small $|t-s|$. These bounds depend on a novel extension of the classical Hardy–Littlewood inequality for fractional integrals, presented in \Cref{hls}. The non-stationarity of the increments requires us to refine traditional methods for deriving sample path properties of Gaussian processes. For example, in \Cref{t}, we establish a general result on the $\Phi$-variation of Gaussian processes, extending Theorem 4 from \cite{kawada1973variation} to processes with non-stationary increments. Additionally, we employ classic techniques such as strong local non-determinism, the Sudakov minoration, and the concentration of measure.

The rest of this paper is organized as follows. In \Cref{main section}, we state our main results, beginning with \Cref{phi} on the $\Phi$-variation of the sample paths, followed by Theorem \ref{modcty} and \ref{unicty}, which address the local and uniform moduli of continuity. In \Cref{dim}, we determine the Hausdorff dimension of the graph of $Y$ as $(2-H)\vee(2-K)$, assuming that $K>2H-1$. \Cref{thm:maximum} establishes that the location of the maximum of $Y$ has an atomless distribution if and only if $H>K$.  \Cref{main section} concludes with an outlook and a list of open questions. \Cref{prelim section}
 provides auxiliary results, some of which may be of independent interest---particularly the extended Hardy--Littlewood inequality given by \Cref{hls}, which plays a central role in our proofs. Finally,  \Cref{main proofs section}  contains the proofs of our main results.

\section{Main results}
\label{main section}

Following \cite{schied2024weierstrass}, let $W_H=(W_H(t))_{t\ge0}$ be the fractional Brownian motion with Hurst parameter $H\in(0,1)$ and starting point $W_H(0)=0$. We pick a deterministic function $\kappa:[0,1]\to[0,1]$ satisfying $\kappa(0)=0$ and $\kappa(1)=1$. The stochastic process 
\begin{align}
B_H(t):=W_H(t)-\kappa(t)W_H(1),\qquad  t\in[0,1],\label{fbb}
\end{align}
can then be regarded as a fractional Brownian bridge with Hurst parameter $H$. For instance, under the choice
\begin{align}\kappa(t):=\frac{1}{2}(1+t^{2H}-(1-t)^{2H}),\label{standard kappa}
\end{align}
 the law of $B_H$ is equal to the law of $W_H$ conditioned on the event $\{W_H(1)=0\}$; see~\cite{GasbarraSottinenValkeila}. However, the specific form of $\kappa$ will not be needed in the sequel. We will only assume that $B_H$ is of the form $B_H(t)=W_H(t)-\kappa(t)W_H(1)$ for some function $\kappa:[0,1]\to[0,1]$ that satisfies $\kappa(0)=0$ and $\kappa(1)=1$ and that is H\"older continuous with some exponent $\tau\in(H,1]$. For example, the function $\kappa$ in \eqref{standard kappa} satisfies these requirements. Both $\kappa$ and $B_H$ will be fixed throughout this paper. We denote by $\fp{x}$  the fractional part of $x\ge0$.

 \begin{definition}
\label{WW}   (\cite{schied2024weierstrass}) For $\alpha\in(0,1)$ and $b\in\{2,3,\dots\}$, the stochastic process
\begin{equation}\label{ww}
Y(t)=Y_{\alpha,b,H}(t):=
\sum_{n=0}^\infty \alpha^nB_H(\fp{b^n t}),\qquad 0\le t\le 1,
\end{equation}
is called the \emph{fractional Wiener--Weierstrass  bridge} with parameters $\alpha$, $b$, and $H$.
\end{definition}

The fractional Wiener--Weierstrass  bridge is a Gaussian process with highly non-stationary increments and, unlike fractional Brownian motion,  is not self-similar. Therefore, many techniques for studying sample path properties are not available. In addition, it has continuous but nowhere-differentiable sample paths, a fractional covariance structure, and is not a semimartingale \cite{schied2024weierstrass}. 

Throughout this paper, we define 
\begin{align}
    K:=\min\Big\{1,\big({-\log_b\alpha}\big)\Big\}.\label{eq:K}
\end{align}
The key message from \cite{schied2024weierstrass} is the following.
\begin{displayquote}
\emph{The roughness of the fractional Wiener--Weierstrass  bridge $\bY=(Y(t))_{t\in[0,1]}$ is determined by a competition between the Hurst exponent $H$ of the underlying fractional Brownian bridge and the roughness exponent $K$ from the Weierstrass-type convolution.}
\end{displayquote}
In other words, the process $\bY$ often inherits sample path properties from the fractional Brownian motion $W_H$ if $H<K$, and from the Weierstrass  fractal construction if $H>K$.  Therefore, the investigation of the sample path properties of $\bY$ often bifurcates, depending on the relation between $K$ and $H$.
For instance, the \textit{roughness} of a function $f:[0,1]\to\R$ can be quantified by the $p$-th variation along the sequence of  $b$-adic partitions, defined as 
\begin{equation}\label{pth variation}
\<f\>^{(p)}_t:=\lim_{n\ua\infty}\sum_{k=0}^{\lfloor tb^n\rfloor}\big|f((k+1)b^{-n})-f(kb^{-n})\big|^p,\qquad t\in[0,1],
\end{equation}
provided the limit exists for all $t$, where $\lfloor x\rfloor$ denotes the largest integer less than or equal to $x$ (see \cite{HanSchiedHurst}). Theorem 2.3 of \cite{schied2024weierstrass} shows that the $p$-th variation of $\bY$ along the sequence of $b$-adic partitions is non-trivial for $p=\min\{1/K,1/H\}$ if $K\neq H$. We refer to Section 2 of \cite{schied2024weierstrass} for further discussions and properties of the fractional Wiener--Weierstrass  bridge.

Our results on the sample path properties of the process $\bY$ will also depend on the interplay between the parameters $H$ and $K$. The most interesting and challenging case is the critical phase $H=K$, where we fully characterize the sample path properties through delicate analysis of the covariance function. 
We start from the (Wiener--Young) $\Phi$-variation, which for a real function $f$ is defined as
\[
v_\Phi(f)=\sup\bigg\{\sum_{i=1}^n\Phi(|f(t_i)-f(t_{i-1})|):\ 0=t_0<t_1<\cdots<t_n=1,\ n\in\bN\bigg\}.
\]
Note that here the supremum is taken over \emph{all} partitions of the interval $[0,1]$ and not just over a specific refining sequence of partitions as for the $p^{\text{th}}$ variation studied in \cite{schied2024weierstrass}.
Our goal is to find a critical function $\Phi$ such that $v_\Phi(\bY)$ 
is non-trivial in the sense that $\bP(0<v_\Phi(\bY)<\infty)=1$. Here and in the sequel, $L$ (resp.~$\delta>0$) will denote a large (resp.~small) number depending only on $\alpha,b,H$ and which may vary at each occurrence. The function $\log\log x$ is always interpreted as $\log\log x\vee e$, where $a\vee b$ and $a\wedge b$ denote the respective maximum and minimum of two real numbers $a$ and $b$.

\begin{theorem}\label{phi}Let $\bY$ be a fractional Wiener--Weierstrass  bridge with parameters $\alpha$, $b$, and $H$, and $K=\min\{1,({-\log_b\alpha})\}$.
\begin{enumerate}[(i)]
\item  If $H<K$, 
\begin{align}
\bP\left(\frac{1}{L}<v_\Phi(\bY)<\infty\right)=1\quad\text{ for }\quad\Phi(x)=\bigg(\frac{x}{\sqrt{2\log\log(1/x)}}\bigg)^{1/H}.\label{p2}
\end{align}
\item  If $H=K$, 
\begin{align}
\bP\left(\frac{1}{L}<v_\Phi(\bY)<\infty\right)=1\quad\text{ for }\quad\Phi(x)=\bigg(\frac{x}{\sqrt{2\log(1/x)\log\log(1/x)/H}}\bigg)^{1/H}.\label{p1}
\end{align}
 \item  If $H>K$, 
\[
\bP(0<v_\Phi(\bY)<\infty)=1\quad\text{ for }\quad\Phi(x)=x^{1/K}.
\]
\end{enumerate}
Moreover, in all three cases, if $\Theta:[0,\infty)\to[0,\infty)$ is a function such that  $\Phi(x)=o(\Theta(x))$ as $x\to 0^+$, then $v_{\Theta}(\bY)=\infty$ almost surely. Conversely, if $\Theta(x)=o(\Phi(x))$ as $x\to 0^+$, then $v_\Theta(Y)=0$ almost surely.
\end{theorem}

\Cref{phi} will follow from a more general result, Theorem \ref{t} below, which extends Theorem 10.3.2 of \cite{MarcusRosen} to the case of Gaussian processes with non-stationary increments. The concept of  $\Phi$-variation plays an important role in  rough path calculus; see, e.g., \cite{friz2020course,friz2010multidimensional}, and the references therein. 
The $\Phi$-variation of various stochastic processes has been the subject of several earlier works: fractional Brownian motion \cite{dudley2011concrete}, bi-fractional Brownian motion \cite{norvaivsa2008varphi}, sub-fractional Brownian motion \cite{qi2018law}, more general Gaussian processes with stationary increments \cite{kawada1973variation}, and  certain non-Gaussian processes \cite{basse2016phi}. 
For  fractional Brownian motion, a critical function $\Phi$ such that $v_{\Phi}(\bW_H)$ is non-trivial is given by $\Phi(x)=x^{1/H}(\log\log(1/x))^{-1/(2H)}$, which coincides with the function $\Phi$ in part (i) of \Cref{phi}. 
 \\

Our next results, Theorems \ref{modcty} and \ref{unicty}, characterize the local and uniform moduli of continuity for fractional Wiener--Weierstrass  bridges. We will write $\R_+=(0,\infty)$. Following Section 7.1 of \cite{MarcusRosen}, we will say that a function $\omega:[0,1)\to[0,\infty)$ with $\omega(0)=0$ is an \textit{exact uniform modulus of continuity} for a Gaussian process $(G(t))_{t\in[0,1]}$ if 
$$\bP\Bigg(\lim_{h\to0}\sup_{\substack{t,s\in[0,1]\\ |t-s|<h}}\frac{|G(t)-G(s)|}{\omega(|t-s|)}=C\Bigg)=1$$
for some constant $C\in\R_+$. We say that a function $\rho:[0,1)\to[0,\infty)$ with $\rho(0)=0$ is an \textit{exact local modulus of continuity} for $(G(t))_{t\in[0,1]}$ at $s\in[0,1]$ if 
$$\bP\left(\limsup_{t\in[0,1],\,t\to s}\frac{|G(t)-G(s)|}{\rho(|t-s|)}=C'\right)=1$$
for some $C'\in\R_+$.

\begin{theorem}\label{modcty}Let $\bY$ be a fractional Wiener--Weierstrass  bridge with parameters $\alpha$, $b$, and $H$, and $K=\min\{1,({-\log_b\alpha})\}$.
\begin{enumerate}[(i)]

\item If $H<K$, then $\rho(u)=u^H\sqrt{\log\log(1/u)}$ is an exact local modulus of continuity for $Y$ at every $s\in[0,1]$.
\item If $H=K$, then $\rho(u)=u^H\sqrt{\log (1/u)\log\log(1/u)}$ is an exact local modulus of continuity for $Y$ at every $s\in[0,1]$.
\item If $H>K$, then there exist non-negative random variables $Z_s,~s\in[0,1]$, such that
\begin{align}
    \bP\left(\limsup_{t\in[0,1],\,t\to s}\frac{|Y(t)-Y(s)|}{|t-s|^{K}}=Z_s\right)=1.\label{eq:Zs}
\end{align}
Moreover, if $\kappa$ is strictly increasing,\footnote{This is a technical assumption, which we expect can be removed.} the random variable $Z_s$ is non-constant, strictly positive, and unbounded for almost every $s\in[0,1]$. 
In particular, $Y$ does not have an exact local modulus of continuity at almost every $s\in[0,1]$.
\end{enumerate}

\end{theorem}

\begin{theorem}\label{unicty}Let $\bY$ be a fractional Wiener--Weierstrass  bridge with parameters $\alpha$, $b$, and $H$, and $K=\min\{1,({-\log_b\alpha})\}$.
\begin{enumerate}[(i)]
    
\item If $H<K$, then $\omega(u)=u^H\sqrt{\log (1/u)}$ is an exact uniform modulus of continuity for $Y$.
\item If $H=K$, then $\omega(u)=u^H\log (1/u)$ is an exact uniform modulus of continuity for $Y$.
\item If $H>K$, then there exists a non-constant and unbounded random variable $Z$ such that 
\begin{align}
    \bP\Bigg(\lim_{h\to0}\sup_{\substack{t,s\in[0,1]\\ |t-s|<h}}\frac{|Y(t)-Y(s)|}{|t-s|^K}=Z\Bigg)=1.\label{eq:Z}
\end{align}
In particular, $Y$ does not have an exact uniform modulus of continuity.
\end{enumerate}

\end{theorem}

It is instructive to compare the moduli of continuity of $\bY$ with those of classical functions or processes. 
First, if $\alpha b>1$, the classical Weierstrass  function \eqref{eq:w} admits local moduli $\rho(u)=u^{-\log_b(\alpha)}$ at all points (Theorem 1 of \cite{izumi1965theoremes}), and hence the same uniform modulus.  
Second, the fractional Brownian motion $W_H$ has an exact uniform modulus $\omega(u)=u^H\sqrt{\log (1/u)}$ and an exact local modulus $\rho(u)=u^H\sqrt{\log\log(1/u)}$, as special cases of Theorems 7.2.14 and 7.6.4 of \cite{MarcusRosen}. In other words, in terms of moduli of continuity, the Wiener--Weierstrass  process $\bY$ mimics the classical Weierstrass-type functions if $H>K$ and the fractional Brownian motion $W_H$ if $H<K$. Surprisingly, in the critical case $H=K$, our modulus of continuity differs from that of the critical Weierstrass  function \eqref{eq:w} with $\alpha b=1$. More precisely, \cite{gamkrelidze1984modulus} established that $\rho(u)=u^H\sqrt{\log (1/u)\log\log\log(1/u)}$ is an exact local modulus of continuity for $\phi_{1/b,b}$ at almost all points. More general results were later established in \cite{de2019central} for (critical) Weierstrass-type functions.

\begin{remark} \label{zer-one remark}
    The fact that $Z_s$ in \eqref{eq:Zs} (or $Z$ in \eqref{eq:Z}) is random does not contradict the well-known zero-one law on the modulus of continuity of Gaussian processes (see Lemma \ref{lemma:0-1} below). In fact, it is an interesting example where the zero-one law fails. This is because for $H>K$ the modulus of continuity is \textit{not} of a larger order than the $L^2$-distance $\nn{Y(t)-Y(s)}_2$  as $|t-s|\to 0$.
\end{remark} 

Denote by $\mathrm{dim}(f)$ the Hausdorff dimension of {the graph} of a function $f$. 
For the Weierstrass  function in \eqref{eq:w} (or Weierstrass-type functions in general), it was a long-standing conjecture that $\mathrm{dim}(\phi_{\alpha,b})=\max\{1,2-K\}$, until recently resolved in \cite{ren2021dichotomy,shen2018hausdorff}. On the other hand, for the fractional Brownian motion we have $\mathrm{dim}(\bW_H)=2-H$ a.s. (Theorem 1 of \cite{adler1977hausdorff}). The question extends naturally to the fractional Wiener--Weierstrass  bridge $\bY$. According to the previous heuristic, it is expected that $\mathrm{dim}(\bY)=\max\{2-H,2-K\}$ holds a.s. We give a partial answer in the next result.

\begin{theorem}\label{dim}
Let $\bY$ be a fractional Wiener--Weierstrass  bridge with parameters $\alpha$, $b$, and $H$, and $K=\min\{1,({-\log_b\alpha})\}$. Suppose that $K>2H-1$. Then $\mathrm{dim}(\bY)=\max\{2-H,2-K\}$ almost surely.
\end{theorem}

A crucial ingredient in analyzing the moduli of continuity and the Hausdorff dimension is estimating the covariance of the fractional Wiener--Weierstrass  bridge $\bY$. Specifically, we need to prove upper and lower bounds for $\nn{\bY(t)-\bY(s)}_2$ for $t,s\in[0,1]$. This non-trivial task will require two technical tools: fractional integral representations and the strong local-nondeterminism property. We will provide the necessary background on fractional integrals in Section \ref{sec:fractional int}. 

The strong local non-determinism is a crucial property for establishing sample path properties of Gaussian processes, such as uniform modulus of continuity \cite{meerschaert2013fernique}, small ball probability, and Chung's law of the iterated logarithm \cite{tudor2007sample}, among many others. See \cite{xiao2006properties,xiao2007strong} for surveys. There are multiple different definitions of local non-determinism for Gaussian processes, among which \cite{berman1973local} and \cite{pitt1978local} contain the earliest versions. One of the most widely used definitions is as follows: we say a Gaussian random field $(X(t))_{t\in I}$ indexed by $I\se \bR^N$ is strongly locally $\rho$-non-deterministic if, for any $n\in\bN$ and any points $u,t_1,\dots,t_n\in  I$,
\begin{align}
\var{X(u)\mid X(t_1),\dots,X(t_n)}\geq \frac{1}{L}\min_{1\leq k\leq n}\rho(u,t_k)^2,\label{slnd}\end{align} where $\rho$ is some prefixed metric on $\bR^N$ and $L>0$ depends only on the law of the process $(X(t))_{t\in I}$. For instance, one may take $\rho(s,t)=|t-s|^{H}$ for the fractional Brownian motion (Lemma 7.1 of \cite{pitt1978local}).  This property also holds for well-known Gaussian processes such as the fractional Brownian sheet \cite{wu2007geometric} and bi-fractional Brownian motion \cite{tudor2007sample}, with certain choices of $\rho$. \\

Finally, we study the location of the maximum of a fractional Wiener--Weierstrass  bridge. Let $\tau_{\alpha,b,H}$ be the (random) location of the maximum of $\bY_{\alpha,b,H}$, taking the leftmost point when the location is not unique, that is, 
\[\tau_{\alpha,b,H}:=\min\Big\{t\in[0,1]:Y_{\alpha,b,H}(t)=\sup_{0\leq s\leq 1}Y_{\alpha,b,H}(s)\Big\}.\]
\begin{theorem}\label{thm:maximum}
Let $K=\min\{1,({-\log_b\alpha})\}$. The distribution of $\tau_{\alpha,b,H}$ is atomless if and only if $H\leq K$. Moreover, $\bP(\tau_{\alpha,b,H}=0)>0$ if $H>K$.
\end{theorem}
The study of the location of the maximum of (deterministic) fractal functions has extensive literature. To mention a few, there are \cite{kahane1959exemple} for the classical Takagi function, \cite{GalkinGalkina,Galkina,MishuraSchied2} for Takagi--Landsberg functions, and \cite{han2022step} that characterizes the set of global maximizers and minimizers for functions in the Takagi class.

\paragraph{Notations.}  We use $\nn{\cdot}_2$ to denote the $L^2(\Omega)$-norm of a generic random variable and $\nn{\cdot}_{L^p}$ to denote the $L^p$-norm of a measurable function in $L^p(\R)$, where $p\geq 1$.  Denote by $\#A$ the cardinality of a finite set $A$, and $|I|$ the Lebesgue measure of a set $I\subseteq \R$.  The set of non-negative integers is denoted by $\bN_0=\bN\cup\{0\}$.

\paragraph{Outlook and some open questions.} We propose a few open questions and interesting directions for future research.
\begin{enumerate}[{\rm 1.}]

\item Several proofs (such as that of \Cref{unicty}(ii)) could be simplified, and further properties of the fractional Wiener--Weierstrass  bridges could perhaps be unveiled, if one can establish the strong local non-determinism \eqref{slnd} for the Wiener--Weierstrass  processes. However, this task appears non-trivial due to the intricate dependence structure created by the fractal construction. We conjecture that if $H<K$,  then $\bY$ is strongly locally $\rho$-non-deterministic with $\rho(s,t)=|s-t|^H$. 

\item As mentioned above, our main results rely heavily on controlling the $L^2$-distance $\nn{Y(t)-Y(s)}_2$. As we will see in Section \ref{c}, obtaining precise bounds of $\nn{Y(t)-Y(s)}_2$ in the case $H>K$ remains a challenging task. On one hand, it is not hard to show that $\nn{Y(t)-Y(s)}_2\leq L|t-s|^{K}$ uniformly in $s,t$. On the other hand, a lower bound of the form
\begin{align}
    \nn{Y(t)-Y(s)}_2\geq \frac{1}{L}|t-s|^{K}\label{eq:lb}
\end{align}
cannot hold uniformly for $s,t\in[0,1]$, because otherwise it would contradict Lemma 7.1.10 of \cite{MarcusRosen} along with Theorem \ref{modcty}(iii). This motivates the following question: for which pairs $(s,t)\in[0,1]^2$ can we have a uniform lower bound \eqref{eq:lb}? Lemma \ref{1l} may serve as a first step by asserting that if $K\in(2H-1,H)$, there exist sets $\{T_N\}_{N\geq 1}$ with Hausdorff dimensions tending to one, such that \eqref{eq:lb} holds uniformly for $t,s\in T_N$, where $L$ may depend on $N$. Extensions of such a result may lead to a better understanding of the local modulus of continuity and the Hausdorff dimension; see the point below.

\item We conjecture that for all $\alpha,b,H$, it holds that $\mathrm{dim}(\bY)=\max\{2-H,2-K\}$ almost surely. Moreover, we conjecture that the random variable $Z_s$ arising in \Cref{modcty}(iii) is non-constant and strictly positive for all $s\in[0,1]$. For this problem, the case $H\in(0,1/2]$ should follow from the Hardy--Littlewood inequality (Lemma \ref{lemma:Hardy--Littlewood inequality} below), while the more challenging case is $H\in(1/2,1)$. Both problems require further investigation of the quantity $ \nn{Y(t)-Y(s)}_2$, particularly the lower bound.  
\end{enumerate}

\section{Some preliminary estimates}
\label{prelim section}

An essential ingredient in proving the main results is estimating the covariance of the fractional Wiener--Weierstrass bridge $\bY$, or in other words, obtaining upper and lower bounds for $\nn{Y(t)-Y(s)}_2$. This will be the goal of the current section. We start with a minimal background on fractional integrals and their connections to moments of fractional Wiener integrals.

\subsection{Background on fractional integration} \label{sec:fractional int}

Let us recall from \cite{Mishura} that the Riemann--Liouville fractional integrals are defined as follows.  For $\beta>0$, 
$$I_+^\beta(f)(x):=\frac{1}{\Gamma(\beta)}\int_{-\infty}^xf(t)(x-t)^{\beta-1}dt\quad\text{ and }\quad I_-^\beta(f)(x):=\frac{1}{\Gamma(\beta)}\int^{\infty}_xf(t)(t-x)^{\beta-1}dt,$$
 $I^0_+(f)(x):=f(x)$, and for $-1<\beta<0$, 
$$I_+^\beta(f)(x):=\frac{1}{\Gamma(1+\beta)}\frac{d}{dx}\int_{-\infty}^xf(t)(x-t)^{\beta}dt\quad\text{ and }\quad I_-^\beta(f)(x):=\frac{-1}{\Gamma(1+\beta)}\frac{d}{dx}\int^{\infty}_xf(t)(t-x)^{\beta}dt.$$
For $H\in(0,1)$, let $\beta=H-1/2$ and define the linear operator
\begin{align*}
    M^H_\pm(f):=
C_HI^{\beta}_\pm(f),
\end{align*}
where $C_H>0$ is chosen as in equation (1.3.3) of \cite{Mishura}. Then,  if $f$ is supported on $[0,\infty)$ and $M^H_-(f)\in L^2(\R)$,
$$\int_0^\infty  f(s)\,dW_H(s)=\int_0^\infty M_-^H(f)(s)\,dW_{1/2}(s),$$
see Section 1.6 of \cite{Mishura}. Thus, by It\^{o}'s isometry,
\begin{align}
    \bE\bigg[\Big|\int_0^\infty  f(s)\,dW_H(s)\Big|^2\bigg]=\int_0^\infty M_-^H(f)(s)^2 ds=\n{M_-^H(f)}_{L^2}^2.\label{eq:second moment formula}
\end{align}
The following result, known as the Hardy--Littlewood inequality, 
provides useful upper and lower bounds of \eqref{eq:second moment formula}.

\begin{lemma}[Corollary 1.9.2 of \cite{Mishura}]\label{lemma:Hardy--Littlewood inequality}
  There exists $L>0$ depending on $H$ such that the following holds.  If $H\in(0,1/2]$, 
    $$ \bE\bigg[\Big|\int_0^\infty  f(s)\,dW_H(s)\Big|^2\bigg]=\n{M_-^H(f)}_{L^2}^2\geq \frac{1}{L}\n{f}_{L^{1/H}}^2.$$
    If $H\in[1/2,1)$,
    $$ \bE\bigg[\Big|\int_0^\infty  f(s)\,dW_H(s)\Big|^2\bigg]=\n{M_-^H(f)}_{L^2}^2\leq L\n{f}_{L^{1/H}}^2.$$
\end{lemma}

\subsection{A Hardy--Littlewood-type inequality for a sum of homogeneous indicator functions}

Our estimates of covariances in \Cref{c} rely on a novel refinement of Lemma \ref{lemma:Hardy--Littlewood inequality} for fractional integrals, which is of independent interest. It applies to a special case of step functions.

\begin{definition}Let $k\in\bN$. An open subset of $[0,\infty)$ is a {\it $k$-interval} if it is a union of at most $k$ bounded open intervals.
\end{definition}

The following theorem will later be applied with $b\in\{2,3,\dots\}$ as in \Cref{WW}. However, the theorem's statement is also true if $b$ is not an integer.

\begin{theorem}\label{hls}
Let $H\in(0,1),\,k\in\bN,\,\alpha\in(0,1),\,b=\alpha^{-1/H}$, and $\{I_m\}_{m\in\bN_0}$ be a collection of $k$-intervals with $|I_m|=b^m$. Define
\begin{align}
f_m:=\alpha^m\bone_{I_m} \quad \text{ and }\quad g_M:=\sum_{m=0}^{M-1}f_m.\label{gm}
\end{align}
Then there exists $L>0$, depending on $k,H,\alpha,b$, such that for all $M\in\bN$,
\begin{align}
    \frac{M}{L}\leq \n{M_-^H(g_M)}_{L^2}^2\leq LM.\label{eq:HL}
\end{align}
\end{theorem}

Let us call a function $f_m$ as in \eqref{gm} a  \textit{homogeneous} indicator function. Theorem \ref{hls} then states that the $L^2$-norm of a fractional integral of a sum of $M$ positive {homogeneous} indicator functions is of order $M$ (up to multiplicative constants).
 We also remark here that if $H\in(0,1/2)$, the classical Hardy--Littlewood inequality (Lemma \ref{lemma:Hardy--Littlewood inequality}) only gives the lower bound
\begin{align}
    \n{M_-^H(g_M)}_{L^2}^2\geq \frac{1}{L}\n{g_M}_{L^{1/H}}^2.\label{eq:HL bad}
\end{align}
By Minkowski's inequality, the right-hand side of \eqref{eq:HL bad} further has the upper bound $\nn{g_M}_{L^{1/H}}^2\leq M^{2H}$, which for large $M$ is strictly dominated by our bound $M/L$. Hence \eqref{eq:HL bad} cannot be optimal. A similar reasoning applies to $H\in(1/2,1)$. Therefore, unless $H=1/2$, our result strengthens the classical Hardy--Littlewood inequality for functions $g_M$.

 
For the proof of \Cref{hls}, we focus first on the lower bound which appears less transparent than the upper bound. The easy part of the lower bound of \eqref{eq:HL} is when $H\geq 1/2$. Indeed, since the increments of fractional Brownian motion are positively correlated for $H\geq 1/2$, we have by \eqref{gm} and \eqref{eq:second moment formula} that
\begin{align*}
    \n{M_-^H(g_M)}_{L^2}^2&= \bE\Bigg[\bigg|\int_0^\infty  \sum_{m=0}^{M-1}f_m(s)\,dW_H(s)\bigg|^2\Bigg]\geq \sum_{m=0}^{M-1}\bE\left[\left|\int_0^\infty  f_m(s)\,dW_H(s)\right|^2\right].
\end{align*}
Next, we write the $k$-interval $I_m$ as the disjoint union of open intervals $J_1,\dots, J_k$ (some of which may be empty). At least one of these intervals say $J_1$, must have length $|I_m|/k=b^m/k$. Hence, using again the positive correlation of the increments of $W_H$,
\begin{align*}
    \bE\left[\left|\int_0^\infty  f_m(s)\,dW_H(s)\right|^2\right]&\ge \alpha^{2m}\sum_{i=1}^k|J_i|^{2H}\ge \alpha^{2m}b^{2mH}k^{-2H}=k^{-2H}.
\end{align*}
Hence, the lower bound in \eqref{eq:HL} holds with $L=k^{-2H}$.

Let us now turn to the case  $H\in(0,1/2)$.  To lighten the notation, we define the bilinear form \[\lr{f}{g}:=\bE\left[\left(\int_0^\infty  fdW_H\right)\left(\int_0^\infty  gdW_H\right)\right]\]on the space of continuous functions with compact support in $[0,\infty)$.

\begin{lemma}\label{l1}
Let $I\subset [0,\infty)$ be a finite union of bounded open intervals, and $h:[0,\infty)\to[0,\infty)$ be a nonnegative step function (with finitely many steps), vanishing on $I^c=[0,\infty)\setminus I$. Then $\lr{\bone_I}{h}\geq 0$.
\end{lemma}

\begin{proof}
We proceed by induction on the number $n$ of disjoint open intervals in $I$. Consider first when $I$ itself is an open interval. By linearity of the expectation, we may assume $h=\bone_J$ is an indicator function of a nonempty interval $J\se I$. Let us denote  $I=(i_1,i_2)$ and $J=(j_1,j_2)$. Then, by the monotonicity of the function $x\mapsto x^{2H}$,
\begin{align*}
\lr{\bone_I}{\bone_J}=\frac{1}{2}(|j_2-i_1|^{2H}+|i_2-j_1|^{2H}-|j_1-i_1|^{2H}-|i_2-j_2|^{2H})\geq 0.
\end{align*}
Next, we assume that the claim holds for $n$ and suppose that $I$ is the disjoint union of $n+1$ nonempty bounded open intervals, denoted by $I_1,\dots,I_{n+1}$. We assume that $I_{n+1}=(i_{n+1},i_{n+2})$ is the rightmost intervals so that $\sup I=i_{n+2}$. We furthermore denote $i_n:=\sup I\setminus I_{n+1}\le i_{n+1}$. By the linearity of the expectation, we may assume without loss of generality that $h$ is the indicator function of some interval $J$ that is contained within some $I_k$. By symmetry and considering the case $k=1$, we may assume that $k\le n$.
Note then that
$g:=\bone_I=g_1+g_2,\text{ where }g_1=\bone_{(i_{n+1},i_{n+2})},\ g_2=\bone_{I\setminus I_{n+1}}$. 
We also define
\[
\widetilde{g}_1(x):=g_1(x+i_{n+1}-i_n),\ \widetilde{g}=\widetilde{g}_1+g_2.
\]
By concavity of the function $x\mapsto x^{2H}$, $\lr{g_1}{h}\geq \lr{\widetilde{g}_1}{h}$. This gives
\begin{align*}
\lr{g}{h}&=\lr{g_1}{h}+\lr{g_2}{h}\geq \lr{\widetilde{g}_1}{h}+\lr{g_2}{h}= \lr{\widetilde{g}}{h},
\end{align*}which is nonnegative by the induction hypothesis and noting that $\widetilde{g}$ is an indicator function of an $n$-interval (modulo a finite collection of points, which does not change the value of $\lr{\widetilde{g}}{h}$).
\end{proof}

\begin{lemma}\label{l2}
Let $0<H<1/2$ and $h:[0,\infty)\to[0,\infty)$ be a nonnegative step function whose support is $I$, which is a finite union of bounded open intervals. Denote by
\begin{align}h_*:=\min h(I)=\min (h([0,\infty))\setminus\{0\}).\label{h}
\end{align}Then for all $\delta\leq h_*$,
\begin{align}
\bE\left[\Big|\int_I hdW_H\Big|^2\right]\geq \frac{\delta^2|I|^{2H}}{L}+\bE\left[\Big|\int_I (h-\delta) dW_H\Big|^2\right].\label{w1}
\end{align}

\end{lemma}

\begin{proof}
Write $I$ as a disjoint union of bounded open intervals $I=I_1\cup I_2\cup\cdots\cup I_n$. Denote by $h_j$ the restriction of $h$ on $I_j$, and $\delta_j=\delta\bone_{I_j},$ so by assumption $g_j:=h_j-\delta_j\geq 0$. 
Thus 
\begin{align*}
\lr{h}{h}&=\lr{\sum_{j=1}^Ng_j}{\sum_{j=1}^Ng_j}+\lr{\sum_{j=1}^N\delta_j}{\sum_{j=1}^N\delta_j}+2\sum_{i=1}^N\sum_{j=1}^N\lr{g_i}{\delta_j}.
\end{align*}
By Lemma \ref{lemma:Hardy--Littlewood inequality} and concavity of the function $x\mapsto x^{2H}$, we have
\[
\lr{\sum_{j=1}^N\delta_j}{\sum_{j=1}^N\delta_j}\geq \frac{\delta^2}{L}\sum_{j=1}^N|I_j|^{2H}\geq \frac{\delta^2|I|^{2H}}{L}.
\]
Thus we conclude by \Cref{l1} that
\[
\lr{h}{h}\geq \frac{\delta^2|I|^{2H}}{L}+\lr{\sum_{j=1}^Ng_j}{\sum_{j=1}^Ng_j}.
\]
Since $\sum_{j=1}^Ng_j=(h-\delta)\bone_I$, we obtain (\ref{w1}).
\end{proof}

We will repeatedly apply \Cref{l2} to bound $\bE\left[|\int_0^\infty g_MdW_H|^2\right]$ from below, and show that the contributions of the form $\delta^2|I|^{2H}/L$ give the desired lower bound. One must show that $h_*$ is large enough at some point in our procedure. This is formulated in the next result.

\begin{lemma}\label{l3}
Let $M\in\bN$ be arbitrary and recall (\ref{gm}). Define $S_M:=g_M(\bR)\se[0,\infty)$. Then there is $L>0$ depending only on $\alpha,k$,\footnote{Recall that each $I_m$ is a $k$-interval.} but not on $M,m$, such that for all $m<M$, $(\alpha^{m+1},\alpha^m)\setminus S_M$ contains an open interval of length $\alpha^m/L$.
\end{lemma}

\begin{proof}
Let us choose $L_1\in\bN$ large with 
\begin{align}
\frac{3kL_1\alpha^{L_1}}{1-\alpha}<\frac{1-\alpha}{2}\label{L11},
\end{align}
which is possible since $\alpha\in(0,1)$.  Since each $I_j$ is a $k$-interval, the function $g_{L_1}$ changes value at most $2k(L_1+1)$ times by the definition \eqref{gm}, showing that $\card g_{L_1}(\bR)\leq 2k(L_1+1)<3kL_1$. Therefore, we may assume without loss of generality that $M>L_1$, because $\card g_{\ell}(\mathbb R)<3kL_1$ for each $\ell\leq L_1$, in which case the statement can be accommodated by increasing $L$, once the assertion is established for $M>L_1$. For $M>L_1$, we consider a {non-negative} integer $m<M-L_1$ and a number $\tau\in (\alpha^{m+1},\alpha^m)\cap S_M$. Then by (\ref{gm}), \[\tau=\sum_{j=m}^{M-1}\ee_{j}\alpha^j,\quad\text{ where }\ee_j\in\{0,1\}.\]
We define $T_{m,L_1}$ to be the collection of  all  vectors $(\ee_m,\dots,\ee_{m+L_1})$. Since $\card g_{L_1}(\bR)<3kL_1$, $\card T_{m,L_1}\leq 3kL_1$. Write {$\{t_1,\dots,t_{\ell}\}$ for the set of numbers that can be represented as $\sum_{j=m}^{m+L_1}\eps_j\alpha^j$ for some $(\ee_m,\dots,\ee_{m+L_1})\in T_{m,L_1}$}. Since $\sum_{i=L_1+1}^\infty\alpha^i=\alpha^{L_1+1}/(1-\alpha)$, 
\begin{align}
\tau\in\bigcup_{j=1}^{\ell}\bigg[t_j,t_j+\sum_{i=m+L_1+1}^\infty\alpha^i\bigg]=\bigcup_{j=1}^{\ell}\bigg[t_j,t_j+\frac{\alpha^{m+L_1+1}}{1-\alpha}\bigg].\label{tauin}\end{align}
By (\ref{L11}) and since $\ell\leq 3kL_1$,
\[
\bigg|\bigcup_{j=1}^{\ell}\bigg[t_j,t_j+\frac{\alpha^{m+L_1+1}}{1-\alpha}\bigg]\bigg|\leq \frac{3kL_1\alpha^{m+L_1+1}}{1-\alpha}<\frac{\alpha^m(\alpha-\alpha^2)}{2},
\]which means that for $L={6kL_1}/({\alpha-\alpha^2})$ there must be an interval\footnote{The location of such an interval may depend on $T_{L_1}$, but the {\it existence} does not.} of length $\alpha^m/L$, 
which is a subset of \[(\alpha^{m+1},\alpha^m)\setminus S_M\supseteq (\alpha^{m+1},\alpha^m)\setminus \bigcup_{j=1}^{3kL_1}\bigg[t_j,t_j+\frac{\alpha^{L_1+1}}{1-\alpha}\bigg].\]
This completes the proof {for $m<M-L_1$. But since $M$ is arbitrary and $L$ is independent of $M$, the case $M-L_1\le m<M$ follows simply by enlarging $M$ and proceeding with the same proof.}
\end{proof}

\begin{proof}
[Proof of Theorem \ref{hls}] The lower bound only needs to be proved for $H<1/2$. The idea  is to repeatedly apply \Cref{l2}, ``shrink'' the function $g_M$ at each step, and extract factors $\delta^2|I|^{2H}/L\geq 1/L$. Denote by $L_2$ the constant in \Cref{l3}. For $0\leq i\leq M-1$, define the shrunk functions of $g_M$ as
\[
g_M^{(i)}:=(g_M-\alpha^{M-i})_+,
\]
which intuitively corresponds to the function $h-\delta$ in  \Cref{l2} after applying this lemma several times. Intuitively, this separates the contributions for different $i$ using the thresholds $\alpha^{M-i-1}$ and $\alpha^{M-i}$. By (\ref{gm}), clearly we have $|g_M^{-1}(\alpha^{M-i},\infty)|\geq b^{M-i-1}$, which implies \begin{align}|\supp g_M^{(i)}|\geq b^{M-i-1}.\label{supp}\end{align}
Let us denote the points in $S_M\cap (\alpha^{M-i-1},\alpha^{M-i})$ by $s_M^{(i,1)}<\cdots<s_M^{(i,N_{M-i})}$, $S_M^{(i,0)}:=\alpha^{M-i-1},~S_M^{(i,N_{M-i}+1)}:=\alpha^{M-i}$, and the truncated functions $g_M^{(i,j)}:=(g_M-s_M^{(i,j)})_+$. By \Cref{l2}, for all $1\leq j\leq N_{M-i}+1$,\[\lr{g_M^{(i,j-1)}}{g_M^{(i,j-1)}}\geq \lr{g_M^{(i,j)}}{g_M^{(i,j)}}+\frac{1}{L}\left(s_M^{(i,j)}-s_M^{(i,j-1)}\right)^2|\supp g_M^{(i,j-1)}|^{2H}.\]Also by \Cref{l3}, there exists $1\leq j_0\leq N_{M-i}+1$ such that $s_M^{(i,j_0)}-s_M^{(i,j_0-1)}\geq \alpha^{M-i-1}/L_2$, so that by (\ref{supp}) and the relation $\alpha b^H=1$,
\begin{align*}
\lr{g_M^{(i)}}{g_M^{(i)}}&\geq \frac{1}{L}\Big(\frac{\alpha^{M-i-1}}{L_2}\Big)^2|\supp g_M^{(i)}|^{2H}+\lr{g_M^{(i+1)}}{g_M^{(i+1)}}\geq \frac{1}{L}+\lr{g_M^{(i+1)}}{g_M^{(i+1)}}.
\end{align*}Summation over $0\leq i\leq M-1$ yields that
\begin{align*}
\n{M_-^H(g_M)}_{L^2}^2&=\bE\left[\Big|\int_0^\infty  g_MdW_H\Big|^2\right]\geq \bE\left[\Big|\int_0^\infty  g_M^{(0)}dW_H\Big|^2\right]\geq \sum_{i=0}^{M-1}\frac{1}{L}=\frac{M}{L}.
\end{align*}
This proves the lower bound.

 Consider now the upper bound. Using the identity
 \begin{align}
     \n{M_-^H(g_M)}_{L^2}^2=\bE\bigg[\Big|\int_0^\infty  g_MdW_H\Big|^2\bigg]=\bE\left[\Big|\sum_{m=0}^{M-1}\int_0^\infty  f_mdW_H\Big|^2\right]=\sum_{m=0}^{M-1}\sum_{k=0}^{M-1}\lr{f_m}{f_k},\label{6}
 \end{align}
 it suffices to bound the right-hand side of \eqref{6} from above. We use induction on $M$. The base case is obvious, i.e., $\lr{f_0}{f_0}\leq L$. Consider for $1\leq n\leq M-1$,
 \begin{align*}
     \sum_{m=0}^{n}\sum_{k=0}^{n}\lr{f_m}{f_k}-\sum_{m=0}^{n-1}\sum_{k=0}^{n-1}\lr{f_m}{f_k}&=\lr{f_n}{f_n}+2\sum_{m=0}^{n-1}\lr{f_n}{f_m}\leq L+L\sum_{m=0}^{n-1}\alpha^{m+n}\lr{\bone_{I_m}}{\bone_{I_n}}.
 \end{align*}
If $0<H\leq 1/2$, then 
 $\lr{\bone_{I_m}}{\bone_{I_n}}\leq b^{2mH}$ due to the negativity of correlations. If $H>1/2$, by mean-value theorem, $\lr{\bone_{I_m}}{\bone_{I_n}}\leq Lb^mb^{n(2H-1)}.$ In either case, one easily checks that 
 \[\sum_{m=0}^{n}\sum_{k=0}^{n}\lr{f_m}{f_k}-\sum_{m=0}^{n-1}\sum_{k=0}^{n-1}\lr{f_m}{f_k}\leq L,\]
 where $L$ depends only on $\alpha,b$. Summing over this relation completes the proof.
 \end{proof}

\subsection{Covariance estimates}\label{c}

In the following results, the notation $\bone_{[c,d]}$ means $-\bone_{[d,c]}$ if $d<c$. Recall \eqref{ww} and \eqref{eq:K}.

\begin{lemma}
Let $\bY$ be the Wiener--Weierstrass  process with $H<K$, then there exists $L>0$ such that
for all $s, t\in[0,1]\text{ satisfying }|s-t|\leq b^{-L}$,
\[ \frac{1}{L}|t-s|^{2H}\leq \bE[|Y(t)-Y(s)|^2]\leq L|t-s|^{2H}.
\]In particular, $\bY$ is a quasi-helix in the sense of \cite{Kahane1981,kahane1993some}.
\label{MC}
\end{lemma}

\begin{proof}
The upper bound is straightforward by Minkowski's inequality and (\ref{ww}), so we focus on the lower bound. Fix $\alpha\in(0,1)$ and $b\in\{2,3,\dots\}$ with $H\leq K$. We consider a large number $L_3$ to be determined and we choose $L\in\bN$ such that
\begin{align}
\sum_{m=L}^\infty \alpha^m<\frac{1}{L_3},\label{m1}
\end{align}
and 
\begin{align}
\begin{cases}
\frac{\alpha^L}{\alpha b^\tau-1}<\frac{1}{L_3}&\text{ if }\alpha b^\tau>1;\\
\forall N\geq L,\ Nb^{-\tau N}<\frac{1}{L_3}&\text{ if }\alpha b^\tau=1;\\
\frac{b^{-\tau L}}{1-\alpha b^\tau}<\frac{1}{L_3}&\text{ if }\alpha b^\tau<1,
\end{cases}\label{L1}
\end{align}
and for all $\delta<b^{-L}$,
\begin{align}
\begin{cases}
L_3\delta^{H+1}<\delta^{2H}&\text{ if }\alpha b^\tau>1;\\
L_3\left(\delta^{H+1}-\delta^2\log\delta\right)<\delta^{2H}&\text{ if }\alpha b^\tau=1;\\
L_3\left(\delta^{H+1}+\delta^2\right)<\delta^{2H}&\text{ if }\alpha b^\tau<1.
\end{cases}\label{L2}
\end{align}

Fix $0\leq s\leq t\leq 1$ with $b^{-M-1}<t-s\leq b^{-M}$ so that $M\geq L$. Observe that
\begin{align*}
Y(t)-Y(s)&=\sum_{m=0}^\infty \alpha^m(B_H(\{b^mt\})-B_H(\{b^ms\}))\\
&=\sum_{m=0}^\infty \alpha^m(W_H(\{b^mt\})-W_H(\{b^ms\})-(\kk(\{b^mt\})-\kk(\{b^ms\}))W_H(1))=\int_0^1g(x)\,dW_H(x)
\end{align*}as a Wiener integral, where\begin{align}
g(x):=\sum_{m=0}^{\infty}\alpha^m\bone_{[\{b^ms\},\{b^mt\}]}(x)-\sum_{m=0}^\infty \alpha^m (\kk(\{b^mt\})-\kk(\{b^ms\})).\label{eq:g}
\end{align}
This integral is well-defined using Lemma 3.1 of \cite{schied2024weierstrass}. Define
\begin{align}
\ell :=\inf\{1\leq k\leq M-1:\{b^ks\}>\{b^kt\}\}\wedge M.\label{k}
\end{align} We claim that for $0\leq k<\ell $,
\begin{align}
0\leq \{b^ks\}<\{b^kt\}<1\ \text{ and }\ \{b^kt\}-\{b^ks\}=b^kt-b^ks\leq b^{k-M},\label{k1}
\end{align}and for $\ell \leq k<M$ (since $t-s<b^{-M}$),
\begin{align}
0\leq \{b^kt\}<\{b^ks\}<1\ \text{ and }\ \{b^ks\}-\{b^kt\}= 1-(b^kt-b^ks)\geq 1-b^{k-M}.\label{k2}
\end{align}
Indeed, $b^\ell t-b^\ell s\leq b^{\ell -M}<1$ implies $\{b^kt\}+(1-\{b^ks\})\leq b^{\ell -M}$ and for $k\in[\ell ,M)$, we have $\{b^kt\}=b^{k-\ell }\{b^kt\}\leq b^{-1}$ and $1-\{b^ks\}=b^{k-\ell }(1-\{b^ks\})\leq b^{-1}$. Therefore, $\{b^kt\}<\{b^ks\}$, i.e., the order of $\{b^kt\}$ and $\{b^ks\}$ flips at most once for $0\leq k<M$ (i.e., if $k=\ell $) and after they flip, one of them is very close to $0$ and the other very close to $1$, with the distances proportional to $b^k$.

Let us write $g(x)=\sum_{i=1}^3g_i(x)$ where
\begin{align}\displaystyle
\begin{cases}\displaystyle
g_1(x):=\sum_{m=0}^{\ell -1}\alpha^m \bone_{[\{b^ms\},\{b^mt\}]}(x)+\sum_{m=\ell }^{M-1}\alpha^m\bone_{[0,\{b^mt\}]\cup[\{b^ms\},1]}(x);\\
g_2(x):=-\displaystyle\sum_{m=0}^{\ell -1}\alpha^m(\kk(\{b^mt\})-\kk(\{b^ms\}))-\sum_{m=\ell }^{M-1}\alpha^m\left(1-(\kk(\{b^ms\})-\kk(\{b^mt\}))\right);\\
g_3(x):=\displaystyle\sum_{m=M}^\infty \alpha^m \bone_{[\{b^ms\},\{b^mt\}]}(x)-\sum_{m=M}^\infty \alpha^m (\kk(\{b^mt\})-\kk(\{b^ms\})).\end{cases}\label{123}
\end{align}
Note that $g_2(x)$ does not depend on $x$.

Let $x\in[s,t]$, then $g_1(x)\geq 1$. We also have
\[
g_3(x)\geq -(1+2\sup|\kk|)\sum_{m=M}^\infty \alpha^m=:-K_1(M).
\]
By (\ref{k2}), for $\ell \leq k<M$, $\{b^ks\}-\{b^kt\}\geq 1-b^{k-M}$, so that by H\"{o}lder continuity of $\kk$,
\begin{align*}
g_2(x)&\geq -\sum_{m=\ell }^{M-1}\alpha^m (\{b^mt\}^\tau+(1-\{b^ms\})^\tau)-\sum_{m=0}^{\ell -1}\alpha^m(b^m(t-s))^{\tau}\\
&\geq -2b^{-M\tau}\sum_{m=0}^{M-1}(\alpha b^\tau)^m\\
&\geq\begin{cases}-2(\alpha b^\tau-1)^{-1}\alpha^M\ &\text{ if } \alpha b^\tau>1\\
-2Mb^{-\tau M}\ &\text{ if } \alpha b^\tau=1\\
-2(1-\alpha b^\tau)^{-1}b^{-\tau M}\ &\text{ if } \alpha b^\tau<1
\end{cases}\quad=:-K_2(M).
\end{align*}
Therefore by (\ref{m1}) and (\ref{L1}), and since $M\geq L$, if $L_3$ is large,
\begin{align}
g(x)\geq 1+g_2(x)+g_3(x)\geq 1-(K_1(M)+K_2(M))\geq \frac{1}{2}.\label{2}
\end{align}
Also by (\ref{L2}) and since $\alpha\leq b^{-H}$, if $L_3$ is large,
\begin{align}
K_1(M)+K_2(M)\leq \frac{1}{2H}|t-s|^{2H-1}.\label{33}
\end{align}
Consider first the case $0<H<1/2$. By Lemma \ref{lemma:Hardy--Littlewood inequality},
\[
\bE[|Y(t)-Y(s)|^2]\geq \frac{1}{L}\n{g}_{L^{1/H}}^2\geq \frac{1}{L}\left(\int_s^tg(x)^{1/H} dx\right)^{2H}\geq \frac{1}{L}|t-s|^{2H},
\]as required. Now let us assume that $1/2\leq H<1$. Expanding the square, we have
\begin{align*}
\bE(|Y(t)-Y(s)|^2)&=\bE\left[\Big|\int_0^1\Big(g-\bone_{[s,t]}\Big)dW_H\Big|^2\right]+\bE[|W_H(t)-W_H(s)|^2]\\
&\hspace{1cm}+2\bE\bigg[(W_H(t)-W_H(s))\bigg(\int_0^1\Big(g-\bone_{[s,t]}\Big)dW_H\bigg)\bigg]\\
&\geq \bE[|W_H(t)-W_H(s)|^2]+2\bE\left[(W_H(t)-W_H(s))\left(-(K_1(M)+K_2(M))W_H(1)\right)\right]\\
&= |t-s|^{2H}-2(K_1(M)+K_2(M))\bE[(W_H(t)-W_H(s))W_H(1)],
\end{align*}
where the second step is because $\bW_H$ has non-negatively correlated increments for $H\geq 1/2$ and (\ref{2}). By the mean-value theorem and since $x\mapsto x^{2H}$ is convex,
\[
\bE[(W_H(t)-W_H(s))W_H(1)]\leq H|t-s|.
\]Thus by (\ref{33}),
\[
\bE[|Y(t)-Y(s)|^2]\geq |t-s|^{2H}-H|t-s|(K_1(M)+K_2(M))\geq \frac{1}{2}|t-s|^{2H}.
\]
This finishes the proof for $H\geq 1/2$.
\end{proof}

Observe that we only picked the interval $[s,t]$ when estimating $g_1(x)$ from below, while ignoring the other contributions from $\{[b^ms,b^mt]:m=1,\dots,L-1\}$. When taking care of this and other possible contributions, a more refined argument can provide a more precise estimate in the case of $H=K$.

\begin{lemma}\label{MC2}
Let $\bY$ be the Wiener--Weierstrass  process with $H=K$, then there exists $L>0$ such that for all $s, t\in[0,1]$ with $|s-t|\leq b^{-L}$,
\[
\frac{1}{L}|t-s|^{2H}\log\Big(\frac{1}{|t-s|}\Big)\leq\bE[|Y(t)-Y(s)|^2]\leq L|t-s|^{2H}\log\Big(\frac{1}{|t-s|}\Big).
\]In particular, $\bY$ is a quasi-helix.
\end{lemma}

\begin{proof}
Consider a large number $L$ and $0\leq s\leq t\leq 1,\ |s-t|\leq b^{-L}$. Choose $M$ such that $b^{-M-1}<t-s\leq b^{-M}$. Recall from the proof of \Cref{MC} that $Y(t)-Y(s)=\int_0^1g(x)\,dW_H(x)$, where $g$ is defined in \eqref{eq:g}. Define $\ell$ as in (\ref{k}) and recall (\ref{123}). Using Minkowski's inequality we will show that the contribution from $g_1$ is of the right order and that from $g_2,g_3$ are negligible. Let us consider first $1/2\leq H<1$ so that the increments are positively correlated, which yields that
\begin{align*}
&\hspace{0.5cm}\bE\bigg[\Big|\int_0^1g_1dW_H\Big|^2\bigg]\\
&\geq \sum_{m=0}^{\ell -1}\bE[\alpha^{2m}(W_H(\{b^mt\})-W_H(\{b^ms\}))^2]+\sum_{m=\ell }^{M-1}\bE[\alpha^{2m}(W_H(\{b^mt\})+W_H(1)-W_H(\{b^ms\}))^2]\\
&\geq \sum_{m=0}^{\ell -1}\alpha^{2m}b^{2mH}|t-s|^{2H}+\frac{1}{L}\sum_{m=\ell }^{M-1}\alpha^{2m}b^{2mH}|t-s|^{2H}\ge \frac1L |t-s|^{2H} M\\
&\geq \frac{1}{L}|t-s|^{2H}(-\log|t-s|),
\end{align*}
where the second inequality follows from the fact that
\begin{align*}
    \bE[(W_H(\{b^mt\})+W_H(1)-W_H(\{b^ms\}))^2]&\geq \max\Big\{\bE[W_H(\{b^mt\})^2],\bE[(W_H(1)-W_H(\{b^ms\}))^2]\Big\}\\
    &\geq \left(\frac{b^m|t-s|}{2}\right)^{2H}.
\end{align*}
On the other hand, to give the upper bound, we further decompose $g_1(x)=\sum_{i=1}^3g_{1,i}(x)$ where
\begin{align}
\begin{cases}
\displaystyle g_{1,1}(x):=\sum_{m=0}^{\ell -1}\alpha^m \bone_{[\{b^ms\},\{b^mt\}]}(x);\\
\displaystyle g_{1,2}(x):=\sum_{m=\ell }^{M-1}\alpha^m\bone_{[0,\{b^mt\}]}(x);\\
\displaystyle g_{1,3}(x):=\sum_{m=\ell }^{M-1}\alpha^m\bone_{[\{b^ms\},1]}(x).\end{cases}\label{3}
\end{align}

 By expanding the square, the mean-value theorem, and since $x\mapsto x^{2H}$ is convex, we have 
\begin{align*}
&\hspace{0.5cm}\bE\bigg[\Big|\int_0^1g_{1,1}dW_H\Big|^2\bigg]\\
&=\sum_{m=0}^{\ell -1}\bE[\alpha^{2m}(W_H(\{b^mt\})-W_H(\{b^ms\}))^2]\\
&\hspace{1.6cm}+2\sum_{0\leq m<k<\ell }\bE[\alpha^{m+k}(W_H(\{b^mt\})-W_H(\{b^ms\}))(W_H(\{b^kt\})-W_H(\{b^ks\}))]\\
&\leq \ell |t-s|^{2H}+L\sum_{0\leq m<k<\ell }\alpha^{m+k}(b^m|t-s|)(b^k|t-s|)^{(2H-1)}\le LM|t-s|^{2H}\\
&\leq L|t-s|^{2H}(-\log|t-s|).
\end{align*}
Similar arguments apply for $g_{1,2},g_{1,3}$. Therefore, by Minkowski's inequality,
\[
\bn{\int_0^1g_1dW_H}_2\leq \sum_{i=1}^3\bn{\int_0^1g_{1,i}dW_H}_2\leq L|t-s|^H\sqrt{-\log|t-s|}.
\]
 Hence we conclude that
\[
\frac{1}{L}|t-s|^{2H}(-\log|t-s|)\leq \bE\bigg[\Big|\int_0^1g_1dW_H\Big|^2\bigg]\leq L|t-s|^{2H}(-\log|t-s|).
\]
That is,
\[\frac{1}{L}|t-s|^{H}\sqrt{-\log|t-s|}\leq \bn{\int_0^1g_1dW_H}_2\leq L|t-s|^{H}\sqrt{-\log|t-s|}\]
We also have by Minkowski's inequality, (\ref{k1}), and (\ref{k2}),
\[
\bn{\int_0^1g_2dW_H}_2=|g_2|\leq |K_2(M)|\leq L|t-s|^H,
\]
where $K_2(M)$ is as in the proof of \Cref{MC} and
\[
\bn{\int_0^1g_3dW_H}_2\leq\sum_{m=M}^\infty \alpha^m \left(\n{W_H(\{b^mt\})-W_H(\{b^ms\})}_2+|\kappa(\{b^mt\})-\kappa(\{b^ms\})|\right)\leq L\alpha^M=L|t-s|^H.
\]
Applying again Minkowski's inequality yields that for $|t-s|<b^{-L}$ and $L$ large enough,
\[\frac{1}{L}|t-s|^{2H}(-\log|t-s|)\leq \bE\bigg[\Big|\int_0^1gdW_H\Big|^2\bigg]\leq L|t-s|^{2H}(-\log|t-s|). 
\]This concludes the case $H\geq 1/2$.

Now we consider the case $0<H<1/2$. Recall (\ref{123}) and (\ref{3}). Following the case $H\geq 1/2$, it suffices to show \[
\frac{1}{L}|t-s|^{2H}(-\log|t-s|)\leq \bE\bigg[\Big|\int_0^1g_1dW_H\Big|^2\bigg]\leq L|t-s|^{2H}(-\log|t-s|).
\]The lower bound now follows from \Cref{hls} (applied with $k=2$) since $b^{-M-1}<t-s\leq b^{-M}$. The upper bound follows similarly as before, which we sketch below: recalling (\ref{3}), we have
\begin{align*}
&\hspace{0.5cm}\bE\bigg[\Big|\int_0^1g_{1,1}dW_H\Big|^2\bigg]\\
&=\sum_{m=0}^{\ell -1}\bE[\alpha^{2m}(W_H(\{b^mt\})-W_H(\{b^ms\}))^2]\\
&\hspace{1.6cm}+2\sum_{0\leq m<k<\ell }\bE[\alpha^{m+k}(W_H(\{b^mt\})-W_H(\{b^ms\}))(W_H(\{b^kt\})-W_H(\{b^ks\}))]\\
&\leq \ell |t-s|^{2H}+L\sum_{0\leq m<k<\ell }\alpha^{m+k}(b^m|t-s|)^{2H}\\
&\leq L|t-s|^{2H}(-\log|t-s|),
\end{align*}
and the rest follows line by line as in the case $H\geq 1/2$. 
\end{proof}

The following result will be useful in deriving the Hausdorff dimension of the graph of fractional Wiener--Weierstrass  bridges.

\begin{lemma}
\label{1l}
Fix $N\in\bN$ and suppose that $K\in(2H-1,H)$. Define
\begin{align}
    T_N:=\{x\in[0,1]:\text{ for all }k\in\bN_0,\ \{b^kx\}\in[b^{-N},1-b^{-N}]\}.\label{eq:TN}
\end{align}
Then for all $t,s\in T_N$ with $|t-s|<b^{-L}$, 
\[\bE[|Y(t)-Y(s)|^2]\geq \frac{1}{L}|t-s|^{2K}.\]
Here, $L$ may depend on $N$.
\end{lemma}

\begin{proof}
We fix $t,s\in T_N$ and $M\in\bN_0$ with $b^{-M-1}<t-s\leq b^{-M}$, where $M>N$ is large and will be determined later. The central fact used here is $\{b^ms\}<\{b^mt\}$ for $0\leq m<M-N$, because otherwise $\{b^mt\}+1-\{b^ms\}=|t-s|b^{m}\leq b^{m-M}<b^{-N}$, contradicting $\{b^ms\},\{b^mt\}\in[b^{-N},1-b^{-N}]$. Note also that for $0\leq m<M-N$, $\{b^mt\}-\{b^ms\}=|t-s|b^m$.

Recall from the proof of \Cref{MC} that 
$Y(t)-Y(s)=\int_0^1g(x)\,dW_H(x)$ where
\[g(x):=\sum_{m=0}^{\infty}\alpha^m\bone_{[\{b^ms\},\{b^mt\}]}(x)-\sum_{m=0}^\infty \alpha^m (\kk(\{b^mt\})-\kk(\{b^ms\})).\]
First,
\begin{align*}\left|\sum_{m=0}^\infty \alpha^m (\kk(\{b^mt\})-\kk(\{b^ms\}))\right|\leq L \sum_{m=0}^{M-N-1} \alpha^m(b^{m}|t-s|)^{H}+L\sum_{m=M-N}^\infty \alpha^m\leq L|t-s|^{-\log_b(\alpha)}.\end{align*}
Similarly,
\[\bigg|\sum_{m=M-N}^\infty \alpha^m\bone_{[\{b^ms\},\{b^mt\}]}(x)\bigg|\leq L|t-s|^{-\log_b(\alpha)}.\]
Here, $L$ depends on $\kappa$, $\alpha$, and $N$, but not on $M$. Thus there exists $L_4>N$ such that 
\[2\,\bigg|-\sum_{m=0}^\infty \alpha^m (\kk(\{b^mt\})-\kk(\{b^ms\}))+\sum_{m=M-N}^\infty \alpha^m\bone_{[\{b^ms\},\{b^mt\}]}(x)\bigg|\leq \alpha^{M-L_4}.\]
Since $L_4$ does not depend on $M$, we may choose $M$ so that $M>L_4$. Then, for $0\leq m\leq M-L_4$ and $x\in[\{b^ms\},\{b^mt\}]$, 
\[g(x)\geq \sum_{k=0}^{M-L_4}\alpha^k\bone_{[\{b^ks\},\{b^kt\}]}(x)-\frac{1}{2}\alpha^m\geq \frac{1}{2}\alpha^m.\]


Let us consider first the case $0<H<1/2$. By Lemma \ref{lemma:Hardy--Littlewood inequality} and our previous observation that $\{b^{M-L_4}t\}-\{b^{M-L_4}s\}=b^{M-L_4}|t-s|$,
\[\bE[|Y(t)-Y(s)|^2]\geq \frac{1}{L}\left(\int_0^1g^{1/H}\,dx\right)^{2H}\geq \frac{1}{L}\left(b^{M-L_4}|t-s|\left(\frac{\alpha^m}{2}\right)^{1/H}\right)^{2H}\geq \frac{1}{L}|t-s|^{-2\log_b(\alpha)}.\]
Now we consider the case $1/2\leq H<1$, where we suppose that $\alpha b^{2H-1}<1$. Consider a large number $L_5>L_4$ to be determined, and let $$h(x):=g(x)-\frac{1}{2}\alpha^{M-L_5}\bone_{[\{b^{M-L_5}s\},\{b^{M-L_5}t\}]}(x).$$
Therefore, by expanding the square,
\begin{align*}
&\hspace{0.5cm}\bE\bigg[\bigg(\int_0^1gdW_H\bigg)^2\bigg]\\
&\geq \bE\bigg[\bigg(\int_0^1\frac{1}{2}\alpha^{M-L_5}\bone_{[\{b^{M-L_5}s\},\{b^{M-L_5}t\}]}dW_H\bigg)^2\bigg]\\
&\hspace{3cm}-\bE\bigg[\bigg(\int_0^1\alpha^{M-L_5}\bone_{[\{b^{M-L_5}s\},\{b^{M-L_5}t\}]}dW_H\bigg)\bigg(\int_0^1hdW_H\bigg)\bigg]\\
&\geq\frac{1}{L}\alpha^{2M-2L_5}b^{2(M-L_5)H}|t-s|^{2H}-\alpha^{M-L_5}\sup|h|\bE[W_H(1)(W_H(\{b^{M-L_5}t\})-W_H(\{b^{M-L_5}s\}))]\\
&\geq \frac{1}{L}\alpha^{2M-2L_5}b^{-2L_5H}-L\alpha^{2M-L_5}b^{-L_5},
\end{align*}
where in the second inequality we used that the increments of $\bW_H$ are positively correlated. Since $\alpha b^{2H-1}<1$, for $L_5$ large enough we have $\alpha^{-2L_5}b^{-2L_5H}>L\alpha^{-L_5}b^{-L_5}$, which yields that
\[\bE\bigg[\bigg(\int_0^1gdW_H\bigg)^2\bigg]\geq \frac{1}{L}\alpha^{2M}= \frac{1}{L}|t-s|^{2K}.\]
This finishes the proof.
\end{proof}

\section{Proofs of the main results}\label{main proofs section}

In this section, we prove Theorem \ref{phi} in Section \ref{s3}, Theorems \ref{unicty}, \ref{modcty}, and \ref{dim} in Section \ref{s2}, and finally Theorem \ref{thm:maximum} in Section \ref{s4}.

\subsection{\texorpdfstring{$\Phi$}{}-variation}\label{s3}

We first prove a general result for the $\Phi$-variation of Gaussian processes, extending Theorem 4 of \cite{kawada1973variation} to processes with non-stationary increments. A corresponding but informal discussion was initiated in Section 10.6 of \cite{MarcusRosen}, while no proofs or precise statements were given. Here, we provide a formal theorem with a detailed proof that requires weaker conditions on the covariance structure.

\begin{theorem}Consider a centered Gaussian process $(X(t))_{t\in[0,1]}$ with
\begin{align}
\frac{1}{L}\sigma(|h|)\leq \n{X(t+h)-X(t)}_2\leq L\sigma(|h|)\label{s}
\end{align}
for all $t,t+h\in[0,1]$, where $\sigma$ is concave and regularly varying at $0$ with index $H\in(0,1)$. Suppose $\Psi(x):=\sigma(x)(\log\log(1/x))^{1/2}$ is strictly increasing and denote by $\Phi$ its inverse, then
\[
\bP\left(\frac{1}{L}<v_\Phi(\bX)<\infty\right)=1
\]for this choice of $\Phi$. \label{t}

\end{theorem}

We first apply the result to prove Theorem \ref{phi}. In reality, it is difficult to write down explicit formulas of $\Phi$ given $\sigma$, but the essential point of $\Phi=\Psi^{-1}$ is to make the following equation \eqref{psi} hold. By Proposition 1.5.15 of \cite{bingham1989regular}, $\Phi$ is regularly varying with index $1/H$ at $0$. Hence, for any $c>0$,
\begin{align}
\lim_{v\to 0}\frac{\Phi(c\Psi(v))}{v}=\lim_{v\to 0}\frac{\Phi(c\Psi(v))}{\Phi(\Psi(v))}=c^{1/H}.\label{psi}\end{align}
In the proof of \Cref{t}, we will only apply $\Phi=\Psi^{-1}$ indirectly through (\ref{psi}).

\begin{proof}[Proof of Theorem \ref{phi}]
Let us first consider the case $H>K$. By Proposition A.1 of \cite{schied2024weierstrass}, the sample paths of $Y$ are H\"{o}lder with exponent $K$, and so $v_\Phi(\bY)<\infty$ a.s. On the other hand,   Theorem 2.3(a) of \cite{schied2024weierstrass} implies that
$
v_\Phi(\bY)>0\text{ a.s.}
$
This proves (iii) of \Cref{phi}. 

Now we turn to parts (i) and (ii). Given the covariance estimates in \Cref{c} and by chopping $[0,1]$ into intervals of length $<b^{-L}$  (because the estimates in \Cref{MC} and \Cref{MC2} apply only for $|s-t|<b^{-L}$), it suffices to prove (\ref{psi}) in the cases $H\leq K$ with our choices of $\Phi$ from \eqref{p2} and \eqref{p1}. For $H<K$ we may refer to (12.42) of \cite{dudley2011concrete}. The case $H=K$ is settled by a direct computation:
\begin{align*}
&\hspace{0.5cm}\lim_{v\to 0}\frac{\Phi_H(c\Psi_H(v))}{v}\\
&=\lim_{v\to 0}\left(\frac{c\sqrt{2v^{2H}\log(1/v)\log\log(1/v)}}{\sqrt{-2\log(c\sqrt{2v^{2H}\log(1/v)\log\log(1/v)})\log(-\log(c\sqrt{2v^{2H}\log(1/v)\log\log(1/v)}))/H}}\right)^{1/H}v^{-1}\\
&=c^{1/H}\lim_{v\to 0}\left(\frac{2\log(1/v)\log\log(1/v)}{-2\log(c\sqrt{2v^{2H}\log(1/v)\log\log(1/v)})\log(-\log(c\sqrt{2v^{2H}\log(1/v)\log\log(1/v)}))/H}\right)^{1/2H}\\
&=c^{1/H}\lim_{v\to 0}\left(\frac{-2\log(v)\log(-\log v)}{-2\log(v^H)\log(-\log(v^H))/H}\right)^{1/2H}\\
&=c^{1/H}.
\end{align*}
This completes the proof.
\end{proof}

\Cref{t} is a  consequence of \Cref{coro:d} below, by following the proof of Corollary 12.23 in \cite{dudley2011concrete}. For a partition  $\kappa=\{0=t_0<t_1<\dots<t_n=1\}$ of $[0,1]$, we denote its mesh by $|\kappa|$ and define
\[
s_\Phi(f,\kappa):=\sum_{i=1}^n\Phi(|f(t_i)-f(t_{i-1})|).
\]
In the following, consider a Gaussian process $\bX$ satisfying the conditions in \Cref{t}.

\begin{theorem}\label{phi2}
Under the above conditions,
\[
\bP\left(\frac{1}{L}\leq\lim_{\delta\to 0}\sup\{s_\Phi(\bX,\kappa):|\kappa|<\delta\}\leq L\right)=1.
\]
\end{theorem}

\begin{corollary}\label{coro:d}
There exists a constant $C\in (0,\infty)$ such that
\[
\bP\left(\lim_{\delta\to 0}\sup\{s_\Phi(\bX,\kappa):|\kappa|<\delta\}=C\right)=1.
\]
\end{corollary}

\begin{proof}
    This follows from Theorem 1 of \cite{kawada1973variation}.
\end{proof}


The proof of \Cref{phi2} will follow the same path as Theorem 12.22 of \cite{dudley2011concrete}. Let us first present some preparatory lemmas. In the statement of these lemmas, we impose the same assumptions as Theorem \ref{t}.

\begin{lemma}\label{p}
For all $y>1$ and all $0\leq h<h+\delta\leq 1$,
\[\bP\bigg(\sup_{s,t\in[h,h+\delta]}|X(s)-X(t)|>Ly\sigma(\delta)\bigg)\leq L(1+L^{-1})^{-y^2}.\]
\end{lemma}

\begin{proof}
Let us first compute $\bE\big[\sup_{t\in[h,h+\delta]}X(t)\big]$. Since $\Psi$ is regularly varying, there is  $\alpha>0$ such that for $k$ large enough, $\Psi(2^{-k-1})\leq 2^{-\alpha}\Psi(2^{-k})$. If $\delta\in[2^{-N},2^{-N+1})$, we have for $N$ large enough, since $\Psi(x)=\sigma(x)\sqrt{\log\log(1/x)}$ is non-decreasing,
\[
\sum_{n=0}^\infty 2^{n/2}\left(\frac{\sigma(\delta 2^{-2^n})}{\sigma(\delta)}\right)\leq \sum_{n=0}^\infty 2^{n/2}\left(\frac{\Psi(2^{-N+1-2^n})\sqrt{\log N}}{\Psi(2^{-N})\sqrt{\log(N+2^n-1)}}\right)\leq \sum_{n=0}^\infty 2^{n/2}2^{\alpha(1-2^n)}\leq L.
\] By (\ref{s}) and Dudley's entropy bound (Proposition 2.5.1 of \cite{talagrand2022upper}), 
\begin{align*}
\bE\bigg[\sup_{t\in[h,h+\delta]}X(t)\bigg]&\leq L\sum_{n=0}^\infty 2^{n/2}\sigma(\delta 2^{-2^n})\leq L\sigma(\delta).
\end{align*}
Lemma 2.2.1 of \cite{talagrand2022upper} and the fact that $\bX$ is symmetric yield that
\[
\bE\bigg[\sup_{s,t\in[h,h+\delta]}|X(s)-X(t)|\bigg]\leq L\sigma(\delta).
\]
Thus by the Gaussian concentration inequality (Theorem 5.4.3 and Corollary 5.4.5 of \cite{MarcusRosen} applied to a dense subset of $[h,h+\delta]$), since $y>1$,
\begin{align*}
&\hspace{0.5cm}\bP\bigg(\sup_{s,t\in[h,h+\delta]}(X(s)-X(t))>Ly\sigma(\delta)\bigg)\\
&\leq \bP\bigg(\Big|\sup_{s,t\in[h,h+\delta]}(X(s)-X(t))-\bE\Big[\sup_{s,t\in[h,h+\delta]}(X(s)-X(t))\Big]\Big|>\frac{1}{2}Ly\sigma(\delta)\bigg)\\
&\leq L\exp\left(\frac{-(\frac{1}{2}Ly\sigma(\delta))^2}{L\sup_{s,t\in[h,h+\delta]}\bE[|X(t)-X(s)|^2]}\right)\\
&\leq L(1+L^{-1})^{-y^2},
\end{align*}as required, where the last step follows from (\ref{s}). This proves the lemma while noting that the absolute value on $X(s)-X(t)$ can be removed by symmetry.
\end{proof}

\begin{lemma}\label{lg}
There exist constants $C(s)\in[\frac{1}{L},L],~s\in[0,1]$ such that
\begin{align}
\bP\left(\limsup_{t\in[0,1],\,t\to s}\frac{|X(t)-X(s)|}{\Psi(|t-s|)}=C(s)\right)=1.\label{local}
\end{align}
In addition, \begin{align}
\bP\left(\limsup_{\delta\to 0}\sup_{|t-s|<\delta}\frac{|X(t)-X(s)|}{\widetilde{\Psi}(|t-s|)}\leq L\right)=1,\label{global}
\end{align}where $\widetilde{\Psi}(x)=\sigma(x)\sqrt{-\log x}$.
\end{lemma}

\begin{proof}
   The proof of \eqref{local} is essentially the same as that of \Cref{modcty} below and hence omitted. Note that $C(s)\in[\frac{1}{L},L]$, otherwise it would contradict Lemma 7.1.10 of \cite{MarcusRosen} with the choice $Y(t)=L^{\pm 2}X(t+(s'-s))$, where $L$ is as in \eqref{s}. The claim \eqref{global} follows from (\ref{s}), Theorem 7.2.1 and Lemma 7.2.5 in \cite{MarcusRosen}.
\end{proof}

\begin{proof}[Proof of \Cref{phi2}] Consider first the lower bound. Denote by $L_6$ the constant in \Cref{lg} so that $C(s)\ge1/L_6$ for all $s$. Let $\ee\in(0,1)$ be arbitrary,\footnote{In fact, one can just take $\ee=1/2$ everywhere.} and
\begin{align}
E(\delta):=\left\{(t,\omega):\ \exists s\in(0,\delta)\cap\bQ,\ \Phi(|X(t+s,\omega)-X(t,\omega)|)>\frac{(1-\ee)s}{L_6^{1/H}}\right\}.\label{delta}
\end{align}
Taking $c=L_6^{-1}$ in (\ref{psi}) yields that there is $\delta$ such that for all $0<s<\delta$,
\[
\Phi\left(\frac{\Psi(s)}{L_6}\right)\geq \frac{(1-\ee)s}{L_6^{1/H}}.
\]
Since $\Phi$ is increasing, (\ref{local}) yields that for each $t\in(0,1)$ the $t$-section $E_t(\delta)$ of $E(\delta)$ is such that $\bP(E_t(\delta))=1$. It follows from Fubini's theorem that $\bP(|E_\delta|=1)=1$. Observe that the set of intervals $[t,t+s]$ with $t\in[0,1]$ and $s$ as in (\ref{delta}) form a Vitali covering of 
\[
E:=\bigcap_{0<\delta\leq 1}E(\delta)=\bigcap_{k=1}^\infty E(1/k).
\]
By Vitali's covering theorem (e.g., Theorem 1 in \S8 of Chapter III in \cite{Natanson}), we may choose a finite subcollection of disjoint intervals $\{[t'_j,t'_j+s_j]\}$ with a total length of at least $1-\ee$. Then for a partition $\kappa$ with mesh $|\kappa|<\delta$,
\[
s_{\Phi}(\bX,\kappa)\geq \sum_{j}\Phi(|X(t'_j+s_j,\omega)-X(t'_j,\omega)|)\geq \frac{1-\ee}{L_6^{1/H}}\sum_j s_j\geq \frac{(1-\ee)^2}{L_6^{1/H}}\geq \frac{1}{L}.
\]

Let us now focus on the upper bound. For any partition $\kappa=\{t_i\}_{i=0}^n$ of $[0,T]$, let $\Delta_i=t_i-t_{i-1}$ and $\Delta_iX:=X(t_i)-X(t_{i-1})$ for $1\leq i\leq n$. Let $I_1,I_2$ be sets of $i\in\{1,\dots,n\}$ such that $|\Delta_iX|$ is respectively in $[0,L_7\Psi(\Delta_i)]$ and $[L_7\Psi(\Delta_i),\infty)$. The constant $L_7$ will be determined later. For any $c,\ee>0$ there is $\eta(c,\ee)>0$ such that for all $v\in(0,\eta(c,\ee))$, $\Phi(c\Psi(v))\leq (c^{1/H}+\ee)v$. Choose $\delta_1=\eta(L_7+\ee,\ee)$, so for any partition $\kappa$ with $|\kappa|<\delta_1$,
\begin{align}
\sum_{i\in I_1}\Phi(|\Delta_iX|)\leq\sum_{i\in I_1}\Phi(L_7\Psi(\Delta_i))\leq (L_7+\ee)^{1/H}\sum_{i\in I_1}\Delta_i\leq L.\label{i1}
\end{align}

To estimate the sum over $I_2$, let $S_{m,j}=[(j/2)e^{-m},(1+j/2)e^{-m}]$ and 
\[
V_{m,j}:=\left\{\omega\in\Omega:\sup_{s,t\in S_{m,j}}|X(t,\omega)-X(s,\omega)|\geq L_7\Psi(e^{-m-2})\right\}.
\]
Let $j_m=\lfloor 2e^m\rfloor-1$. We have
\[
\card\Lambda_m:=\card \{i\in I_2:e^{-m-1}<2\Delta_i\leq e^{-m}\}\leq 5\card\{0\leq j\leq j_m:\omega\in V_{m,j}\}=:Z_m(\omega).
\]
Denote by $L_8$ the constant in \Cref{p}. Take $L_9$ with $(1+L_8^{-1})^{L_9^2e^{-8H}}\geq e^{4+2/H}$ and $L_7=L_8L_9$, we obtain by \Cref{p} with $\delta=e^{-m}$ and $y=L_9\Psi(e^{-m-2})/\sigma(e^{-m})$ that for $m$ large enough,
 \[\bP(V_{m,j})\leq L_8(1+L_8^{-1})^{-y^2}\leq L(m+2)^{-4-2/H},\]
thus $\bE[Z_m]\leq Le^mm^{-4-2/H}$, so that by Markov's inequality and the Borel--Cantelli lemma, with probability one there is $m_1=m_1(\omega)>3$  such that for all $m>m_1(\omega),\ Z_m(\omega)\leq e^mm^{-2-2/H}$.

Since $\sigma$ is regularly varying with index $H$, using Karamata's representation we write
\[
\sigma(x)=x^H\beta(x)\exp\left(\int_1^x\frac{\alpha(u)}{u}du\right),
\]
where $\alpha(u)\to 0,\beta(x)\to C\neq 0$ as $x\to 0$. It is then easy to see that
\begin{align}
\frac{\Psi(e^{-m}m^{2/H})}{\Psi(e^{-m})}\geq \frac{\sigma(e^{-m}m^{2/H})}{\sigma(e^{-m})}\geq m\label{2rv}
\end{align}
for $m$ large enough, say $m>m_0$.

 By (\ref{global}), there exists $K(\omega)<\infty$ such that $|X(s)-X(t)|\leq K(\omega)\widetilde{\Psi}(|t-s|)$ for all $s,t\in[0,1]$ and for almost every $\omega$. Pick $m_2(\omega)=\max\{m_0, m_1(\omega)\}$ and define $\delta_2=\delta_2(\omega)=e^{-m_2(\omega)}\wedge \eta(K(\omega),\ee)$. Then for $\kappa$ with $|\kappa|<\delta_2$, each $m\geq  m_2(\omega)$, and each $i\in \Lambda_m$, there is $j$ such that $[t_{i-1},t_i]\se S_{m,j}$. Thus by (\ref{2rv}) applied on the second line,
\begin{align}
\begin{split}
    \sum_{i\in I_2}\Phi(|\Delta_iX|)&=\sum_{m\geq m_2(\omega)}\sum_{i\in I_2\cap\Lambda_m}\Phi(|\Delta_iX|)\\
&\leq \sum_{m\geq m_2(\omega)}Z_m(\omega)\Phi(K(\omega)\widetilde{\Psi}(e^{-m}))\\
&\leq \sum_{m\geq m_2(\omega)}e^mm^{-2-2/H}\Phi(K(\omega)\Psi(e^{-m}m^{2/H}))\\
&\leq \sum_{m\geq m_2(\omega)}m^{-2}(K(\omega)^{1/H}+\ee)\leq L.
\end{split}\label{i3}
\end{align}
Combining (\ref{i1}) and (\ref{i3}) yields that with probability one, for any partition $\kappa$ with $|\kappa|<\delta_1\wedge\delta_2$,
\[
s_\Phi(X,\kappa)\leq \sum_{j=1}^2\sum_{i\in I_j}\Phi(|\Delta_iX|)\leq L,
\]
completing the proof of the upper bound. Lastly, we note that the final assertion involving the function $\Theta$ is obvious.
\end{proof}

Finally, we remark that the upper bound part of \Cref{t} may also follow from Theorem 1.3 of \cite{basse2016phi} along with regular variation techniques in \cite{bingham1989regular}. On the other hand, \Cref{t} is stronger in the sense that it characterizes the critical $\Phi$ for which the $\Phi$-variation is non-trivial.

\subsection{Moduli of continuity and Hausdorff dimension}\label{s2}

We will frequently use the following zero-one law on the moduli of continuity for Gaussian processes.

\begin{lemma}[Lemma 7.1.1 of \cite{MarcusRosen}]\label{lemma:0-1}
    Let $(G(t))_{t\in[0,1]}$ be a centered Gaussian process for which $d(t,s):=\bE[(G(t)-G(s))^2]$ is continuous. Let furthermore   $\omega,\rho:[0,1)\to[0,\infty)$ be continuous  functions with $\omega(0)=\rho(0)=0$, and $K\in[0,1]$ be a compact set. Then 
    $$\lim_{h\to 0}\sup_{\substack{t,s\in K\\ |t-s|<h}}\frac{G(t)-G(s)}{\omega(|t-s|)}\leq L\quad\text{ a.s. }\implies ~\lim_{h\to 0}\sup_{\substack{t,s\in K\\ |t-s|<h}}\frac{G(t)-G(s)}{\omega(|t-s|)}=L'\quad\text{ a.s. }$$
    for some $L'\geq 0$. Similarly, for each $s\in[0,1]$,
   $$\limsup_{t\in[0,1],\,t\to s}\frac{G(t)-G(s)}{\rho(|t-s|)}\leq L_s\quad\text{ a.s. }\implies ~\limsup_{t\in[0,1],\,t\to s}\frac{G(t)-G(s)}{\rho(|t-s|)}=L_s'\quad\text{ a.s. }$$
    for some $L_s'\geq 0$.
\end{lemma}

\begin{proof}[Proof of \Cref{modcty}]
(i) This follows from \Cref{MC} and Theorem 7.6.4 of \cite{MarcusRosen}, applied with $\phi(h)=h^{2H}$.

(ii) This follows from \Cref{MC2} and Theorem 7.6.4 of \cite{MarcusRosen}, applied with $\phi(h)=h^{2H}(-\log h)$.

(iii) By pathwise H\"{o}lder continuity (Proposition A.1 of \cite{schied2024weierstrass}), for each $s\in[0,1]$,
$$\limsup_{t\in[0,1],\,t\to s}\frac{|Y(t)-Y(s)|}{|t-s|^{K}}<\infty\quad\text{ a.s.},$$
and hence the random variable $Z_s$ is well-defined. It remains to show that $Z_s$ is strictly positive non-constant, and unbounded for almost every $s\in[0,1]$. 
We fix a large integer $L_{10}>0$ to be determined, and consider the set of $s\in[0,1]$ such that there exists an infinite sequence $n_k\to\infty$ such that for all $k$, $\{sb^{n_k-L_{10}}\},\{sb^{n_k-L_{10}+1}\},\dots,\{sb^{n_k}\}\in[0,1/3]$. This condition holds for almost every $s\in[0,1]$ by considering the $b$-adic decimal expansion.

Let
$$X_n=\alpha^{-n}\big(Y(b^{-n}+s)-Y(s)\big).$$
We claim that 
\begin{align}
    \limsup_{k\to\infty}\E[|X_{n_k}|^2]>0.\label{eq:5}
\end{align}
Suppose that \eqref{eq:5} holds. Since each $X_{n_k}$ is Gaussian, for each $x\geq 0$, there exists $\delta_x>0$ such that
$$\bP\bigg(\limsup_{t\in[0,1],\,t\to s}\frac{|Y(t)-Y(s)|}{|t-s|^{K}}>x\bigg)\geq \bP\Big(\limsup_{k\to\infty}|X_{n_k}|>x\Big)\geq \limsup_{k\to\infty}\bP(|X_{n_k}|>x)\geq \delta_x>0,$$
where the second inequality follows from Fatou's lemma. 
Taking $x=0$ yields $Z_s>0$ a.s., and taking $x\to\infty$ yields that $Z_s$ is non-constant and unbounded, as desired.

It remains to establish \eqref{eq:5}. Using \eqref{ww}, we write
$$X_n=\sum_{m=0}^\infty \alpha^{m-n}\big(B_H(\{b^{m-n}+b^ms\})-B_H(\{b^ms\})\big)=\sum_{m=1}^n\alpha^{-m}\big(B_H(\{b^{-m}+b^{n-m}s\})-B_H(\{b^{n-m}s\})\big).$$
By Minkowski's inequality and since $\alpha b^H>1$, we have
\begin{align}
    \bn{\sum_{m=L_{10}+1}^{n} \alpha^{-m}\big(B_H(\{b^{-m}+b^{n-m}s\})-B_H(\{b^{n-m}s\})\big)}_2\leq \sum_{m=L_{10}+1}^{n} \alpha^{-m}b^{-mH}\leq L (\alpha b^H)^{-L_{10}}.\label{eq:m1}
\end{align}
On the other hand, for $n=n_k$, we write
\begin{align*}
    &\hspace{0.5cm}\sum_{m=1}^{L_{10}}\alpha^{-m}\big(B_H(\{b^{-m}+b^{n_k-m}s\})-B_H(\{b^{n_k-m}s\})\big)\\
    &=\sum_{m=1}^{L_{10}}\alpha^{-m}\big(B_H(b^{-m}+\{b^{n_k-m}s\})-B_H(\{b^{n_k-m}s\})\big)\\
    &=\sum_{m=1}^{L_{10}}\alpha^{-m}\big(W_H(b^{-m}+\{b^{n_k-m}s\})-W_H(\{b^{n_k-m}s\})\big)\\
    &\hspace{3cm}-\sum_{m=1}^{L_{10}}\alpha^{-m}\big(\kappa(b^{-m}+\{b^{n_k-m}s\})-\kappa(\{b^{n_k-m}s\})\big)W_H(1)\\
    &=:\int_0^1 \bigg(\psi_{n_k}(t)-\sum_{m=1}^{L_{10}}\alpha^{-m}\big(\kappa(b^{-m}+\{b^{n_k-m}s\})-\kappa(\{b^{n_k-m}s\})\big)\bigg) dW_H(t).
\end{align*}
Since $\kappa$ is strictly increasing, there is a constant $\delta>0$ independent of $L_{10}$ such that 
$$\sum_{m=1}^{L_{10}}\alpha^{-m}\big(\kappa(b^{-m}+\{b^{n_k-m}s\})-\kappa(\{b^{n_k-m}s\})\big)>\delta.$$
Moreover, by construction, each $\psi_{n_k}$ is supported on $[0,5/6]$. By the strong local non-determinism of $W_H$, we have 
$$\bn{\int_0^1 \bigg(\psi_{n_k}(t)-\sum_{m=1}^{L_{10}}\alpha^{-m}\big(\kappa(b^{-m}+\{b^{n_k-m}s\})-\kappa(\{b^{n_k-m}s\})\big)\bigg) dW_H(t)}_2\geq \delta$$
for some $\delta>0$. Altogether, by Minkowski's inequality and  \eqref{eq:m1}, we conclude that
\begin{align*}
    \n{X_{n_k}}_2&\geq \delta-L((\alpha b^H)^{-L_{10}}+\alpha^{L_{10}}).
\end{align*}
Since $\delta,L$ do not depend on $L_{10}$, picking $L_{10}$ large enough yields \eqref{eq:5}, completing the proof.
\end{proof}

Suppose that $H\leq K$. In this case, 
\Cref{modcty} is an immediate consequence of covariance estimates from Section \ref{c}. On the other hand, covariance estimates do not suffice for establishing Theorem \ref{unicty}. The following new strategy will be needed: we consider a large $n$ and compare the Wiener--Weierstrass  process $\bY$ as defined in (\ref{ww}) with a new process \begin{align}X^{(n)}(t):=\sum_{m=0}^{n-1}\alpha^mW_H(\{b^mt\}).\label{Xn}\end{align}One easily sees that by the uniform modulus of continuity of $\bW_H$, the sum converges as $n\to\infty$ if $H<K$. Hence, when studying the limit behavior of the process $\bY$, it suffices to consider $\bX^{(n)}$ because their difference is of a small scale. A few changes will be needed in the critical case $H=K$ where one can only truncate the series but cannot replace the bridge $\bB_H$ by the fractional Brownian motion $\bW_H$.

The following technical lemma establishes the uniform modulus of continuity of $\bX^{(n)}$. Note that when approximating $\bY$ with $\bX^{(n)}$, the number $n$ depends on the scale we look at, which motivates the choice of $N_n$ below.

\begin{lemma}
\label{xnuni}
If $H<K$ we define $\bX^{(n)}$ as in (\ref{Xn}), and $\ee_n=\alpha^n,\ N_n=nH,\ \delta_n=b^{-nH}$.\footnote{Without loss of generality, we assume $N_n$ is always an integer, by using instead $\lfloor nH\rfloor$ or $\lceil nH\rceil$. } Then we have
\begin{align}
\sum_{n=1}^\infty\bP\Bigg(\sup_{\substack{t,s\in[0,1]\\ |t-s|<\ee_n/\delta_n}}\big|X^{(N_n)}(t)-X^{(N_n)}(s)\big|\leq \frac{1}{L}\left(\frac{\ee_n}{\delta_n}\right)^{H}\sqrt{-\log (\ee_n/\delta_n)}\Bigg)<\infty.\label{haha}\end{align}
\end{lemma}

\begin{proof} The plan is to apply the Sudakov minoration (e.g. Lemma 2.10.2 of \cite{talagrand2022upper}) to the increments of $\bX^{(N_n)}$ on well-spaced subintervals of $[0,1]$. To achieve this goal, we first need to estimate the covariances. We consider two large constants $L_{11},L_{12}$ to be determined later, and for each $n$ consider subintervals of length $\ee_n/(L_{11}\delta_n)$ of $[1/b,2/3]$ that are at least $L_{12}\ee_n/(2L_{11}\delta_n)$ apart. We consider the collection of intervals $$\left\{\left[\frac{1}{b}+\frac{jL_{12}\ee_n}{L_{11}\delta_n},\frac{1}{b}+\frac{(jL_{12}+1)\ee_n}{L_{11}\delta_n}\right]\right\}_{1\leq j\leq j_n},$$
where $j_n$ is as large as possible such that the intervals lie in $[1/b,2/3]$. One easily sees that $j_n\geq \delta_n/(L\ee_n)$. Let us define for $1\leq j\leq j_n$,
\begin{align*}&\Delta^\bX_{n,j}:=X^{(N_n)}\left(\frac{1}{b}+\frac{(jL_{12}+1)\ee_n}{L_{11}\delta_n}\right)-X^{(N_n)}\left(\frac{1}{b}+\frac{jL_{12}\ee_n}{L_{11}\delta_n}\right),\\
&\Delta^\bW_{n,m,j}:=W_H\left(b^m\left(\frac{1}{b}+\frac{(jL_{12}+1)\ee_n}{L_{11}\delta_n}\right)\right)-W_H\left(b^m\left(\frac{1}{b}+\frac{jL_{12}\ee_n}{L_{11}\delta_n}\right)\right).\end{align*}
We then have for $1\leq k<\ell\leq j_n$, for some measurable function $f_n$,
\begin{align*}
    &\hspace{0.5cm}\bE\left[\left|\Delta^\bX_{n,k}-\Delta^\bX_{n,\ell}\right|^2\right]\\
    &=\bE\bigg[\bigg|\sum_{m=0}^{N_n-1}\alpha^m(\Delta^\bW_{n,m,k}-\Delta^\bW_{n,m,\ell})\bigg|^2\bigg]\\
    &=\bE\bigg[\bigg|W_H\left(\frac{1}{b}+\frac{kL_{12}\ee_n}{L_{11}\delta_n}\right)-W_H\left(\frac{1}{b}+\frac{(kL_{12}+1)\ee_n}{L_{11}\delta_n}\right)+\int_{\frac{1}{b}+\frac{(kL_{12}+1)\ee_n}{L_{11}\delta_n}}^\infty f_n(x)\,dW_H(x)\bigg|^2\bigg]\\
    &\geq \frac{1}{L}\left(\frac{\ee_n}{L_{11}\delta_n}\right)^{2H},
\end{align*}where the last step follows from the strong local non-determinism applied to $\bW_H$ (Lemma 7.1 of \cite{pitt1978local}). Thus by the Sudakov minoration (e.g., Lemma 2.10.2 of \cite{talagrand2022upper}),
\[ \bE\bigg[\sup_{1\leq j\leq j_n}\Delta^\bX_{n,j}\bigg]\geq\frac{1}{L}\Big(\frac{L_{12}\ee_n}{L_{11}\delta_n}\Big)^{2H}\log(j_n)\geq \frac{1}{L}\Big(\frac{\ee_n}{\delta_n}\Big)^{H}\sqrt{-\log (\ee_n/\delta_n)}. \]
Moreover, for any $j$,
\[\n{\Delta^\bX_{n,j}}_2^2\leq L\Big(\frac{\ee_n}{\delta_n}\Big)^{2H},\]
 which follows from the negativity of covariances if $0< H<1/2$ and Lemma \ref{lemma:Hardy--Littlewood inequality} if $1/2\leq H<1$. By the Gaussian concentration inequality (Theorem 5.4.3 and Corollary 5.4.5 of \cite{MarcusRosen}),
 \begin{align*}
     \bP\Bigg(\bigg|\sup_{1\leq j\leq j_n}\Delta^\bX_{n,j}-\bE\bigg[\sup_{1\leq j\leq j_n}\Delta^\bX_{n,j}\bigg]\bigg|\geq u\Bigg)\leq L\exp\left(\frac{-u^2}{L(\ee_n/\delta_n)^{2H}}\right).
 \end{align*}
 Taking $u=L(\ee_n/\delta_n)^H(-\log (\ee_n/\delta_n))^{1/4}\ll (\ee_n/\delta_n)^H\sqrt{-\log (\ee_n/\delta_n)}$ shows that\[\sum_{n=1}^\infty\bP\left(\sup_{1\leq j\leq j_n}\Delta^\bX_{n,j}\leq \frac{1}{L}(\frac{\ee_n}{\delta_n})^{H}\sqrt{-\log (\ee_n/\delta_n)}\right)<\infty.\] Then (2) holds from the trivial bound
 \[\sup_{\substack{t,s\in[0,1]\\ |t-s|<\ee_n/\delta_n}}\left(X^{(N_n)}(t)-X^{(N_n)}(s)\right)\geq\sup_{1\leq j\leq j_n}\Delta^\bX_{n,j}.\]
This finishes the proof.
\end{proof}

\begin{proof}[Proof of \Cref{unicty}]
(i) Consider the sequence of numbers $\ee_n,\delta_n,N_n$ as in \Cref{xnuni} and denote by $\phi(n)=\ee_n^H\sqrt{-\log\ee_n}$, and we have by \Cref{xnuni} (using here $H<K$),
\[\sum_{n=1}^\infty\bP\Bigg(\sup_{\substack{t,s\in[0,1]\\ |t-s|<\ee_n/\delta_n}}\big|X^{(N_n)}(t)-X^{(N_n)}(s)\big|\leq \frac{1}{L}\left(\frac{\ee_n}{\delta_n}\right)^{H}\sqrt{-\log \ee_n}\Bigg)<\infty.\]
Note that $\bX^{(N_n)}$ is $H$-self-similar, so that
\begin{align}\sum_{n=1}^\infty\bP\Bigg(\sup_{\substack{t,s\in[0,\delta_n]\\ |t-s|<\ee_n}}\big|X^{(N_n)}(t)-X^{(N_n)}(s)\big|\leq \frac{1}{L}\phi(n)\Bigg)<\infty.\label{ts}\end{align}
Due to the continuity of $\bW_H$, there is an a.s. finite random variable $K=K(\omega)$ such that
\[\max\bigg\{W_H(1), \sup_{t\in[0,1]}B_H(t)\bigg\}<K.\]
Consider a number $L_{13}>0$ and the event $E_1:=\{\max\{W_H(1), \sup_{t\in[0,1]}B_H(t)\}<L_{13}\}$, which has positive probability for all $L_{13}$. Thus on this event, by the triangle inequality, for $t<\delta_n$,
\begin{align}\begin{split}
    |Y(t)-X^{(N_n)}(t)|
&\leq \sum_{k=N_n}^\infty \alpha^kB_H(\{b^kt\})+\sum_{k=0}^{N_n-1}\alpha^k|W_H(b^kt)-B_H(b^kt)|\\
&\leq LL_{13}\alpha^{N_n}+L_{13}\sum_{k=0}^{N_n-1}\alpha^k(b^k\delta_n)^\tau=o(\phi(n)).
\end{split}\label{yxn}\end{align}
It follows from (\ref{ts}), (\ref{yxn}), the triangle inequality, and the Borel--Cantelli lemma that on the event $E_1$, eventually almost surely
\[\sup_{\substack{t,s\in[0,\delta_n]\\ |t-s|<\ee_n}}|Y(t)-Y(s)|> \frac{1}{L}\phi(n).\]
We then have
\[\bP\Bigg(\lim_{n\to\infty}\sup_{\substack{t,s\in[0,\delta_n]\\ |t-s|<\ee_n}}|Y(t)-Y(s)|> \frac{1}{L}\phi(n)\Bigg)\geq \bP(E_1)>0.\]Since $\ee_n$ forms a geometric sequence, we can extend the limit to a continuous one. That is,
\begin{align}
    \bP\Bigg(\lim_{h\to 0}\sup_{\substack{t,s\in[0,1]\\ |t-s|<h}}\frac{|Y(t)-Y(s)|}{h^H\sqrt{-\log h}}> \frac{1}{L}\Bigg)\geq \bP(E_1).\label{eq:3}
\end{align}
Recall from \Cref{MC} that $|t-s|^H/L\leq \n{Y(t)-Y(s)}_2\leq L|t-s|^H$. Theorem 7.2.1 of \cite{MarcusRosen} then implies that 
\begin{align}
    \bP\Bigg(\lim_{n\to\infty}\sup_{\substack{t,s\in[0,1]\\ |t-s|<h}}\frac{Y(t)-Y(s)}{h^H\sqrt{-\log h}}\leq L\Bigg)=1.\label{eq:4}
\end{align}
Lemma \ref{lemma:0-1} together with \eqref{eq:3} and \eqref{eq:4} then complete the proof. \\

(ii) Consider the scale $\ee_n=\alpha^n=b^{-nH}$ and $N_n=nH$. We consider three large numbers $L_{14},L_{15},L_{16}$ to be determined later,\footnote{We will later see that $L_{15}$ depends on $L_{16}$ and $L_{14}$ depends on $L_{15},L_{16}$.} and for each $n$ consider a number $j_n$ and for $1\leq j\leq j_n$ define
\[s_{n,j}:=b^{-n(H-L_{14}^{-1})-1}+\frac{jL_{16}\ee_n}{L_{15}}\quad\text{ and }\quad  t_{n,j}:=b^{-n(H-L_{14}^{-1})-1}+\frac{(jL_{16}+1)\ee_n}{L_{15}}=s_{n,j}+\frac{\ee_n}{L_{15}}.\]We pick  the largest integer $j_n$ such that $t_{n,j_n}<b^{-n(H-L_{14}^{-1})}$. Observe that \[j_n\geq \frac{1}{L}b^{-n(H-L_{14}^{-1})+nH}=\frac{1}{L}b^{L_{14}^{-1}n}.\]
For intuition, the reader may compare this with the proof of \Cref{xnuni}, the intervals $[s_j,t_j]$ are now ``well-spaced" subintervals of $[b^{-n(H-L_{14}^{-1})-1},b^{-n(H-L_{14}^{-1})}]$.

Consider the truncated Wiener--Weierstrass  process \[Y^{(n)}(t):=\sum_{m=0}^{n-1}\alpha^mB_H(\{b^mt\}).\]Define for $1\leq j\leq j_n,$
\begin{align}\Delta^\bY_{n,j}:=Y^{(N_n)}(t_{n,j})-Y^{(N_n)}(s_{n,j})\quad\text{ and }\quad \Delta^\bB_{n,m,j}:=B_H(\{b^mt_{n,j}\})-B_H(\{b^ms_{n,j}\}).\label{yb}\end{align}
Let $M_n:=n(H-L_{14}^{-1})$, so that for $0\leq m\leq M_n$, $\Delta^\bB_{n,m,j}=B_H(b^mt_{n,j})-B_H(b^ms_{n,j})$.

By definition and Minkowski's inequality, for $1\leq k<\ell\leq j_n$,
\begin{align}\begin{split}
    \n{\Delta^\bY_{n,k}-\Delta^\bY_{n,\ell}}_2&=\bn{\sum_{m=0}^{N_n-1}\alpha^m(\Delta^\bB_{n,m,k}-\Delta^\bB_{n,m,\ell})}_2\\
&\geq \bn{\sum_{m=0}^{M_n-1}\alpha^m(\Delta^\bB_{n,m,k}-\Delta^\bB_{n,m,\ell})}_2-\bn{\sum_{m=M_n}^{N_n-1}\alpha^m\Delta^\bB_{n,m,k}}_2-\bn{\sum_{m=M_n}^{N_n-1}\alpha^m\Delta^\bB_{n,m,\ell}}_2.
\end{split}\label{7}\end{align}
Our first goal is to prove the lower bound
\[\bE\big[|\Delta^\bY_{n,k}-\Delta^\bY_{n,\ell}|^2\big]\geq \frac{1}{L}\ee_n^{2H}(-\log\ee_n),\]
where $L$ may depend on $L_{14},L_{15},L_{16}$. We first bound $\nn{\sum_{m=M_n}^{N_n-1}\alpha^m\Delta^\bB_{n,m,k}}_2$ from above. Consider the sets  $$T_n:=\{m\in[M_n,N_n-1]\cap\bZ:\ \{b^ms_{n,j}\}\leq\{b^mt_{n,j}\}\}\quad\text{ and }\quad S_n:=[M_n,N_n-1]\cap \bZ\setminus T_n.$$Using the bridge relation \eqref{fbb} and Minkowski's inequality, we write
\begin{align*}
    &\hspace{0.5cm}\bn{\sum_{m=M_n}^{N_n-1}\alpha^m\Delta^\bB_{n,m,j}}_2\\
    &=\bn{\sum_{m=M_n}^{N_n-1}\alpha^m(W_H(\{b^mt_{n,j}\})-W_H(\{b^ms_{n,j}\})-(\kappa(\{b^mt_{n,j}\})-\kk(\{b^ms_{n,j}\}))W_H(1))}_2\\
    &\leq \bn{\sum_{m\in T_n}\alpha^m(W_H(\{b^mt_{n,j}\})-W_H(\{b^ms_{n,j}\}))}_2+\bn{\sum_{m\in T_n}\alpha^m(\kappa(\{b^mt_{n,j}\})-\kk(\{b^ms_{n,j}\}))W_H(1)}_2\\
    &\hspace{3cm}+\bn{\sum_{m\in S_n}\frac{\alpha^mb^m\ee_n}{L_{15}}(W_H(\{b^mt_{n,j}\})+W_H(1)-W_H(\{b^ms_{n,j}\}))}_2\\
    &\hspace{3cm}+\bn{\sum_{m\in S_n}\alpha^m(\kappa(\{b^mt_{n,j}\})+1-\kk(\{b^ms_{n,j}\}))W_H(1)}_2.
    \end{align*}
The second term is bounded by
\[\bn{\sum_{m=M_n}^{N_n-1}\alpha^m(\kappa(\{b^mt_{n,j}\})-\kk(\{b^ms_{n,j}\}))W_H(1)}_2\leq \sum_{m=M_n}^{N_n-1}\alpha^m\left(\frac{b^m\ee_n}{L_{15}}\right)^\tau\leq L\ee_n^H.\]
To estimate the first term, we define 
\[f_n(t):=\sum_{m\in T_n}\alpha^m\bone_{[\{b^ms_{n,j}\},\{b^mt_{n,j}\}]}(t),\]
so that 
\[\bn{\sum_{m=M_n}^{N_n-1}\alpha^m(W_H(\{b^mt_{n,j}\})-W_H(\{b^ms_{n,j}\}))}_2^2=\bE\left[\left|\int_0^\infty f_n(t)\,dW_H(t)\right|^2\right].\]
By \Cref{hls} and self-similarity of $\bW_H$,
\[\bE\left[\left|\int_0^\infty f_n(t)\,dW_H(t)\right|^2\right]\leq L(N_n-M_n)\ee_n^{2H}\leq \frac{Ln\ee_n^{2H}}{L_{14}}\]
where $L$ does not depend on $L_{14}$. Similar estimates hold for the third and fourth terms using the case $k=2$ of \Cref{hls}. We then conclude that
\begin{align}
    \n{\Delta^\bY_{n,k}-\Delta^\bY_{n,\ell}}_2\geq \bn{\sum_{m=0}^{M_n-1}\alpha^m(\Delta^\bB_{n,m,k}-\Delta^\bB_{n,m,\ell})}_2-\frac{L\sqrt{n}\ee_n^H}{L_{14}}.\label{d1}
\end{align}

Next, we bound $\nn{\sum_{m=0}^{M_n-1}\alpha^m(\Delta^\bB_{n,m,k}-\Delta^\bB_{n,m,\ell})}_2$ from below. We first define $\Delta^\bW_{n,m,k}$ similarly as in (\ref{yb}), so by a similar Minkowski's inequality argument, it suffices to give a lower bound for $$\bn{\sum_{m=0}^{M_n-1}\alpha^m(\Delta^\bW_{n,m,k}-\Delta^\bW_{n,m,\ell})}_2.$$ This is done again by an induction argument. Consider first the term with $m=0$. We compute 
\begin{align*}
    \bE[|\Delta^\bW_{n,0,k}-\Delta^\bW_{n,0,\ell}|^2]&=\bE[|\Delta^\bW_{n,0,k}|^2]+\bE[|\Delta^\bW_{n,0,\ell}|^2]-2\bE[\Delta^\bW_{n,0,k}\Delta^\bW_{n,0,\ell}]\\
    &\geq 2\left(\frac{\ee_n}{L_{15}}\right)^{2H}-2L\left(\frac{\ee_n}{L_{15}}\right)^{2}\left(\frac{L_{16}\ee_n}{L_{15}}\right)^{2H-2}=\bigg(\frac{2}{L_{15}^{2H}}-\frac{2LL_{16}^{2H-2}}{L_{15}^{2H}}\bigg)\ee_n^{2H},
\end{align*}where the last line follows from the mean-value theorem and the number $L$ does not depend on $L_{15},L_{16}$. Picking $L_{16}$ large enough we see that 
\begin{align}\bE\big[|\Delta^\bW_{n,0,k}-\Delta^\bW_{n,0,\ell}|^2\big]\geq \frac{\ee_n^{2H}}{LL_{15}^{2H}}.\label{r0}\end{align}
Consider the case $1\leq M\leq M_n-1$. By self-similarity,
\[\bE\big[|\alpha^M(\Delta^\bW_{n,M,k}-\Delta^\bW_{n,M,\ell})|^2\big]\geq \frac{\ee_n^{2H}}{LL_{15}^{2H}}.\]
 Thus we have
\begin{align}
    \begin{split}
        &\hspace{0.5cm}\bE\bigg[\Big|\sum_{m=0}^{M}\alpha^m(\Delta^\bW_{n,m,k}-\Delta^\bW_{n,m,\ell})\Big|^2\bigg]-\bE\bigg[\Big|\sum_{m=0}^{M-1}\alpha^m(\Delta^\bW_{n,m,k}-\Delta^\bW_{n,m,\ell})\Big|^2\bigg]\\
    &=\bE\left[\left|\alpha^M(\Delta^\bW_{n,M,k}-\Delta^\bW_{n,M,\ell})\right|^2\right]+2\bE\bigg[\alpha^M(\Delta^\bW_{n,M,k}-\Delta^\bW_{n,M,\ell})\Big(\sum_{m=0}^{M-1}\alpha^m(\Delta^\bW_{n,m,k}-\Delta^\bW_{n,m,\ell})\Big)\bigg]\\
    &\geq \frac{\ee_n^{2H}}{LL_{15}^{2H}}+2\alpha^M\bE\bigg[(\Delta^\bW_{n,M,k}-\Delta^\bW_{n,M,\ell})\Big(\sum_{m=0}^{M-1}\alpha^m(\Delta^\bW_{n,m,k}-\Delta^\bW_{n,m,\ell})\Big)\bigg].
    \end{split}\label{r1}
\end{align}
Consider first $0<H\leq 1/2$ where the increments of $\bW_H$ are negatively correlated, then
\begin{align}
    \begin{split}
        &\hspace{0.5cm}\bE\bigg[(\Delta^\bW_{n,M,k}-\Delta^\bW_{n,M,\ell})\Big(\sum_{m=0}^{M-1}\alpha^m(\Delta^\bW_{n,m,k}-\Delta^\bW_{n,m,\ell})\Big)\bigg]\\
    &\geq \bE\bigg[\Delta^\bW_{n,M,k}\Big(\sum_{m=0}^{M-1}\alpha^m\Delta^\bW_{n,m,k}\Big)\bigg]+\bE\bigg[\Delta^\bW_{n,M,\ell}\Big(\sum_{m=0}^{M-1}\alpha^m\Delta^\bW_{n,m,\ell}\Big)\bigg]\\
    &=\sum_{m=0}^{M-1}\alpha^m\bE\left[\Delta^\bW_{n,M,k}\Delta^\bW_{n,m,k}+\Delta^\bW_{n,M,\ell}\Delta^\bW_{n,m,\ell}\right].
    \end{split}\label{r2}
\end{align}
Observe that for $m<M$ and $1\leq j\leq j_n$,  $b^mt_{n,j}<b^Ms_{n,j}-b^M/L$, so that by mean-value theorem (recalling $t_{n,j}-s_{n,j}=\ee_n/L_{15}$)
\begin{align}|\bE[\Delta^\bW_{n,M,k}\Delta^\bW_{n,m,k}]|\leq L\Big(\frac{b^M}{L}\Big)^{2H-2}b^{m+M}\frac{\ee_n^2}{L_{15}^2},\label{r3}\end{align}
where $L$ does not depend on $L_{15}$. Using the relation $M<nH$ and $\ee_n=\alpha^n=b^{-nH}$, and combining (\ref{r1}), (\ref{r2}), and (\ref{r3}), we have for $L_{15}$ large enough,
\[\bE\bigg[\Big|\sum_{m=0}^{M}\alpha^m(\Delta^\bW_{n,m,k}-\Delta^\bW_{n,m,\ell})\Big|^2\bigg]-\bE\bigg[\Big|\sum_{m=0}^{M-1}\alpha^m(\Delta^\bW_{n,m,k}-\Delta^\bW_{n,m,\ell})\Big|^2\bigg]\geq \frac{\ee_n^{2H}}{L}.\]
By summing over the previous relation and using (\ref{r0}), we obtain
\[\bE\bigg[\Big|\sum_{m=0}^{M}\alpha^m(\Delta^\bW_{n,m,k}-\Delta^\bW_{n,m,\ell})\Big|^2\bigg]\geq \frac{M\ee_n^{2H}}{L}\geq \frac{1}{L}\ee_n^{2H}(-\log\ee_n).\]
A similar argument works for $H>1/2$. Combining with (\ref{d1}) and \eqref{7}, we obtain that for $L_{14}$ large enough,
\[\n{\Delta^\bY_{n,k}-\Delta^\bY_{n,\ell}}_2\geq \frac{1}{L}\ee_n^H\sqrt{-\log\ee_n}.\]
Therefore, by  the Sudakov minoration,
\[ \bE\bigg[\sup_{1\leq j\leq j_n}\Delta^\bY_{n,j}\bigg]\geq \frac{1}{L}\left(\frac{1}{L}\ee_n^H\sqrt{-\ee_n}\right)\sqrt{\log j_n}\geq \frac{1}{L}\ee_n^H(-\log\ee_n). \]

Our next goal is to bound $\nn{\Delta^\bY_{n,j}}_2$ from above. Using Minkowski's inequality and \Cref{MC2},
\[\n{\Delta^\bY_{n,j}}_2\leq \n{Y(t_{n,j})-Y(s_{n,j})}_2+\n{\widetilde{Y}^{(N_n)}(t_{n,j})-\widetilde{Y}^{(N_n)}(s_{n,j})}_2\leq Ln\ee_n^H+\n{\widetilde{Y}^{(N_n)}(t_{n,j})-\widetilde{Y}^{(N_n)}(s_{n,j})}_2,\]where \[\widetilde{Y}^{(N_n)}(t)=\sum_{m=N_n}^\infty \alpha^mB_H(\{b^mt\}).\]By a similar argument as in (\ref{yxn}) and Gaussian concentration,
\[\n{\widetilde{Y}^{(N_n)}(t_{n,j})-\widetilde{Y}^{(N_n)}(s_{n,j})}_2\leq L\ee_n^H,\]and hence, $\nn{\Delta^\bY_{n,j}}_2\leq Ln\ee_n^H$. By the Gaussian concentration inequality,
\begin{align*}
     \bP\Bigg(\bigg|\sup_{1\leq j\leq j_n}\Delta^\bY_{n,j}-\bE\bigg[\sup_{1\leq j\leq j_n}\Delta^\bY_{n,j}\bigg]\bigg|\geq u\Bigg)\leq L\exp\left(\frac{-u^2}{Ln\ee_n^{2H}}\right).
 \end{align*}
Choose $H'\in(\max\{2H,1\},2)$. Taking $u=\ee_n^H(-\log \ee_n)^{H'}$ and using the trivial bound
\[\sup_{\substack{t,s\in[0,1]\\ |t-s|<\ee_n}}\left(Y^{(N_n)}(t)-Y^{(N_n)}(s)\right)\geq \sup_{1\leq j\leq j_n}\Delta^\bY_{n,j}\]yield that
\[\sum_{n=1}^\infty \bP\Bigg(\sup_{\substack{t,s\in[0,1]\\ |t-s|<\ee_n}}\left(Y^{(N_n)}(t)-Y^{(N_n)}(s)\right)\leq\frac{1}{L}\ee_n^H(-\log \ee_n)\Bigg)<\infty.\]
Consider a large number $L_{17}$ and the non-trivial event $$E_2:=\left\{\sup_{0\leq t\leq 1}B_H(t)\leq L_{17}\right\}.$$Thus on this event,
\[\left|\left(Y^{(N_n)}(t)-Y^{(N_n)}(s)\right)-\left(Y(t)-Y(s)\right)\right|\leq LL_{17}\ee_n^H.\]
By the Borel--Cantelli lemma, on $E_2$ we have eventually a.s.
\[\sup_{\substack{t,s\in[0,1]\\ |t-s|<\ee_n}}|Y(t)-Y(s)|> \frac{1}{L}\ee_n^H(-\log \ee_n).\]
Since $\ee_n=\alpha^n$ forms a geometric sequence, we obtain
\begin{align}
    \bP\Bigg(\lim_{n\to\infty}\sup_{\substack{t,s\in[0,1]\\ |t-s|<h}}\frac{Y(t)-Y(s)}{h^H(-\log h)}>\frac{1}{L}\Bigg)\geq\bP(E_2)>0.\label{eq:1}
\end{align}
Recall from \Cref{MC2} that $|t-s|^H\sqrt{-\log|t-s|}/L\leq \n{Y(t)-Y(s)}_2\leq L|t-s|^H\sqrt{-\log|t-s|}$. Theorem 7.2.1 of \cite{MarcusRosen} then implies that 
\begin{align}
    \bP\Bigg(\lim_{n\to\infty}\sup_{\substack{t,s\in[0,1]\\ |t-s|<h}}\frac{Y(t)-Y(s)}{h^H(-\log h)}\leq L\Bigg)=1.\label{eq:2}
\end{align}
Lemma \ref{lemma:0-1} together with \eqref{eq:1} and \eqref{eq:2} then complete the proof. \\

(iii) This follows directly from  \Cref{modcty}(iii). 
\end{proof}

\begin{proof}[Proof of \Cref{dim}]
The upper bound of the Hausdorff dimension follows from pathwise H\"{o}lder continuity (Proposition 2.3 of \cite{falconer2004fractal} and Proposition A.1 of \cite{schied2024weierstrass}). To show the lower bound, first note that by \Cref{MC}, if $H<K$, there exists $L\in\bN$ such that
$$\bE\left[\left|\frac{Y(t)-Y(s)}{|t-s|^{H}}\right|^2\right]\geq \frac{1}{L}\quad\text{ for }\quad(s,t)\in [0,1]^2,\,|s-t|<b^{-L}.$$ Thus by Gaussianity, there exists some $a>0$ such that for all $0\leq k<b^L$,
\[
\sup_{s,t\in[kb^{-L},(k+1)b^{-L}]}\bP\left(-x\leq \frac{Y(t)-Y(s)}{|t-s|^{H}} \leq x\right)\leq a x\quad\text{ as }\quad x\to 0^+.
\]
By Theorem 2 of \cite{kono1986hausdorff}, the Hausdorff dimension of the graph $\bY$ on $[kb^{-L},(k+1)b^{-L}]$ is bounded from below by $2-H$ almost surely. By countable stability of the Hausdorff dimension (see Section 2.2 of \cite{falconer2004fractal}), we must have $\mathrm{dim}(\bY)\geq 2-H$ on $[0,1]$. The case $H=K$ is similar using \Cref{MC2} instead of \Cref{MC}.

The case $H>K$ requires more work. Recall from \eqref{eq:K} that $K=\min\{1,({-\log_b\alpha})\}=-\log_b\alpha$, so it suffices to prove for a fixed $\ee>0$ that $\mathrm{dim}(\bY)\geq 2+\log_b(\alpha)-2\ee$. Let us fix $b\in\bN$, $H\in(0,1)$, $\alpha\in (0,1)$ with $\alpha b^{2H-1}<1$, and $\ee>0$. 
 By the potential-theoretic lower bound (see Section 4.3 of \cite{falconer2004fractal}), it suffices to find a probability measure $\nu=\nu_\omega$ on the graph $G_\omega$ of $\bY$ such that
\begin{align}
    I_\omega:=\int_{G_\omega}\int_{G_\omega}\frac{d\nu(t)\,d\nu(s)}{|t-s|^\gamma}<\infty\quad\text{ a.s.}\label{eq:IO}
\end{align}  
Recall \eqref{eq:TN}. Let $N$ be a large even number such that $b^{N(1-\ee)}<b^N-2$. Define a set $S_N\subseteq [0,1]$ by its $b^{N/2}$-adic expansion, such that $x=\sum_{i=1}^\infty \xi_i b^{-iN/2}\in S_N$ if $\xi_i\in\{1,2,\dots,b^{N/2}-2\}$ for all $i$. 
It is elementary to check that $S_N\subseteq T_N$. Moreover, the set $S_N$ in \Cref{1l} is a Cantor-like self-similar set. We may define the uniform measure on $S_N$, similarly to the construction of the Cantor measure (see Example 17.1 of \cite{falconer2004fractal}). This can be done, for instance, by choosing each decimal $\xi_i$ independently and uniformly from $\{1,\dots,b^{N/2}-2\}$. Denote such a measure by $\mu_N$. It is then standard to verify that if $b^{N(1-\ee)/2}<b^{N/2}-2$, then
\begin{align}\int_{T_N}\int_{T_N}|t-s|^{\ee-1}d\mu_N(t)\,d\mu_N(s)<\infty;\label{TN}\end{align}
see, for instance, Exercise 4.11 of \cite{falconer2004fractal} for the Cantor set.

We let $\nu$ be the pushforward of $\mu_N$ under the map $t\mapsto (t,X(t,\omega))$, thus reducing \eqref{eq:IO} to proving that 
\[\bE[I_\omega]=\int_{T_N}\int_{T_N}\bE\left[\left(|t-s|^2+|Y(t)-Y(s)|^2\right)^{-\gamma/2}\right]d\mu_N(t)\,d\mu_N(s)<\infty.\] By \Cref{1l} and a similar computation as in the proof of Theorem 2(i) of \cite{kono1986hausdorff}, we obtain
\[\bE\left[\left(|t-s|^2+|Y(t)-Y(s)|^2\right)^{-\gamma/2}\right]\leq L|t-s|^{-\gamma+1+\log_b(\alpha)}\leq L|t-s|^{\ee-1}.\]
Combining with (\ref{TN}) yields $\E[I_\omega]<\infty$, and hence the proof is complete.
\end{proof}

\subsection{Distribution of the maximum location}\label{s4}

\begin{proof}[Proof of Theorem \ref{thm:maximum}]
Suppose first $H\leq K$, we prove that for all $s\in[0,1]$, $\bP(s\text{ is a local maximum})=0$. By arguing similarly as in the proof of (b) and (c) of \Cref{modcty} and applying Lemma \ref{lemma:0-1}, for $H=K$, there exists a deterministic function $C_1:[0,1]\to\bR_+$ such that for all $s\in[0,1]$, 
\begin{align*}
\bP\left(\liminf_{t\in[0,1],\,t\to s}\frac{Y(t)-Y(s)}{\sqrt{2|t-s|^{2H}\log(1/|t-s|)\log\log(1/|t-s|)}}=-C_1(s)\right)=1,
\end{align*}
and for $H<K$, there exists $C_2:[0,1]\to\bR_+$ such that
\[
\bP\left(\liminf_{t\in[0,1],\,t\to s}\frac{Y(t)-Y(s)}{\sqrt{2|t-s|^{2H}\log\log (1/|t-s|)}}=-C_2(s)\right)=1.
\]
This along with parts (b) and (c) of \Cref{modcty} proves $\bP(s\text{ is a local maximum})=0$.

For the remainder of the proof, let $H>K$. We prove that $\bP(0\text{ is a global maximum})>0$. We make an auxiliary claim that there is a Lipschitz function $\phi:[0,1]\to\bR$ such that\footnote{$\phi(t)=1-\cos(2\pi t)$ may work, but we take the shortest path.}
\begin{itemize}
\item $\phi(0)=\phi(1)=0,\phi(t)> 0$ for all $t\in(0,1)$. In particular, $\phi$ attains its minimum at $0,1$;
\item $\phi\in\mathcal H_H$ where $\mathcal H_H$ is the Cameron--Martin space of the fractional Brownian motion with parameter $H$.
\end{itemize}
Provided the claim is true, we consider a large number $L$ and let \[f(t):=L\sum_{m=0}^\infty \alpha^m\phi(\{b^mt\})\] be the fractal function associated with $L\phi$. It is easy to see that $f(t)$ attains its minima at $0,1$, and\footnote{Of course, the $L$ here may not be the same and may depend on the choice of $\phi$.} 
\begin{align}f(t)\geq L(t\wedge (1-t))^K.\label{li2001gaussian}
\end{align}
Consider the process \[\widetilde{Y}(t):=Y(t)+f(t)=\sum_{m=0}^\infty \alpha^m \widetilde{B}_H(\{b^mt\}),\] so that $Y(t)=\widetilde{Y}(t)-f(t)$. It suffices to show $\widetilde{Y}(t)$ is H\"{o}lder continuous with exponent $-\log_b(\alpha)$ and constant $<1/L$ with positive probability, because this implies $\bP(|\widetilde{Y}(t)|\leq\frac{1}{L}(t\wedge(1-t))^K)>0$, which along with (\ref{li2001gaussian}) yields $\bP(Y(t)\leq 0,\forall t\in[0,1])>0$. 

We recall in the proof of Proposition A.1 of \cite{schied2024weierstrass} and the uniform modulus of continuity of fractional Brownian bridge that there exists $\delta_1>0$ such that we have the inclusion of events
\[\left\{  \sup_{t\in[0,1]}|\widetilde{B}_H(t)|\leq\delta_1\right\}\se\left\{\widetilde{Y}(t)\text{ is }-\log_b(\alpha)\text{-H\"{o}lder with constant }\frac{1}{L}\right\}.\]
By the bridge relation and since $\widetilde{B}_H(t)=B_H(t)+L\phi(t)$, there is $\delta_2>0$ depending on $\kappa$ such that
\[\left\{  \sup_{t\in[0,1]}|W_H(t)+L\phi(t)|\leq\delta_2\right\}\se\left\{  \sup_{t\in[0,1]}|B_H(t)+L\phi(t)|\leq\delta_1\right\}\se\left\{  \sup_{t\in[0,1]}|\widetilde{B}_H(t)|\leq\delta_1\right\}.\]
Since $\phi\in\mathcal H_H$, so does $-L\phi$. This implies the probability of the event on the left-hand side is positive by Theorem 2 of \cite{kuelbs1994gaussian} (or by the Cameron--Martin Theorem for the fractional Brownian motion). Combining the above gives $$\bP\big(\widetilde{Y}(t)\text{ is }-\log_b(\alpha)\text{-H\"{o}lder continuous with constant }1/L\big)>0.$$

We now prove the auxiliary claim. The function $\phi$ will arise from $\phi_1$ of the form
\[\phi_1(t):=\int_0^t(t-s)^{H-1/2}\psi_1(s)\,ds\]
where $\psi_1(s)=s(1-s)$. It is easy to see that there is $T>0$ with $\phi_1(0)=\phi_1(T)=0$ and $\phi_1$ is positive and Lipschitz on $[0,T]$. Define $\psi(s):=\psi_1(Ts)$ and \[\phi(t):=\int_0^t(t-s)^{H-1/2}\psi(s)\,ds,\]by a change of variable we obtain the first item of the claim; that $\phi\in\mathcal H_H$ follows from Theorem 5.4 of \cite{picard2011representation} and the fact that $\psi$ is square-integrable.
\end{proof}

 \parskip-0.5em\renewcommand{\baselinestretch}{0.9}\normalsize
\bibliography{CTbook}{}
\bibliographystyle{plain}

\end{document}